\documentclass[final]{siamart190516}

\usepackage{amsmath,amsfonts,amssymb}
\usepackage{xcolor}
\usepackage[algo2e, ruled, linesnumbered]{algorithm2e}
\usepackage{doi}
\usepackage{epstopdf}
\usepackage{afterpage}
\usepackage[protrusion=true,tracking=false,kerning=true]{microtype}
\usepackage{booktabs}

\newcommand{\revision}[1]{#1}

\newcommand{\be}{\begin{enumerate}}
\newcommand{\ee}{\end{enumerate}}
\newcommand{\bi}{\begin{itemize}}
\newcommand{\ei}{\end{itemize}}

\newcommand{\rr}{\mathbb{R}}

\newcommand{\qq}{\mathbb{Q}}

\newcommand{\R}{\mathbb{R}}
\newcommand{\bS}{\mathbb{S}}
\newcommand{\Oh}{\mathcal{O}}

\newcommand{\Sc}{\Sigma^\circ}      
\newcommand{\Ssc}{(\Sigma^*)^\circ} 
\newcommand{\defeq}{\ensuremath{\overset{\mathrm{def}}{=}}}  

\newcommand{\cC}{\mathcal{C}}
\newcommand{\cP}{\mathcal{P}}

\newcommand{\cV}{\mathcal{V}}

\newcommand{\vA}{\mathbf{A}}
\newcommand{\vB}{\mathbf{B}}
\newcommand{\vd}{\mathbf{d}}

\newcommand{\vM}{\mathbf{M}}

\newcommand{\vp}{\mathbf{p}}
\newcommand{\vr}{\mathbf{r}}
\newcommand{\vq}{\mathbf{q}}
\newcommand{\vs}{\mathbf{s}}
\newcommand{\vS}{\mathbf{S}}
\newcommand{\vt}{\mathbf{t}}
\newcommand{\vu}{\mathbf{u}}
\newcommand{\vv}{\mathbf{v}}
\newcommand{\vw}{\mathbf{w}}
\newcommand{\vx}{\mathbf{x}}
\newcommand{\vy}{\mathbf{y}}
\newcommand{\vz}{\mathbf{z}}
\newcommand{\vzero}{\mathbf{0}}
\newcommand{\vone}{\mathbf{1}}
\newcommand{\T}{\mathrm{T}}
\newcommand{\vxo}{{\vx_{1}}}

\newtheorem{example}{Example}

\newcommand{\qedhere}{{}}

\newcommand{\TheTitle}{Dual certificates and efficient rational sum-of-squares decompositions for polynomial optimization over compact sets}
\newcommand{\TheAuthors}{Maria M. DAVIS and D{\'a}vid PAPP}

\headers{Dual sum-of-squares certificates}{\TheAuthors}

\title{{\TheTitle}\thanks{Submitted on \today.\funding{This material is based upon work supported by the National Science Foundation under Grant No.~DMS-1719828 and Grant No.~DMS-1847865.}}
}
\author{
	Maria M. DAVIS\thanks{North Carolina State University, Department of Mathematics. Email: \email{mlmacaul@ncsu.edu}.}
	\and
	D{\'a}vid PAPP\thanks{North Carolina State University, Department of Mathematics. Email: \email{dpapp@ncsu.edu}.}
}

\ifpdf
\hypersetup{
	pdftitle={\TheTitle},
	pdfauthor={\TheAuthors}
}
\fi 

\newcommand{\alittlelessspace}{\vspace{-0.9ex}}

\begin{document}
\maketitle

\begin{abstract}
We study the problem of computing weighted sum-of-squares (WSOS) certificates for positive polynomials over a compact semialgebraic set. Building on the theory of interior-point methods for convex optimization, we introduce the concept of dual certificates, which allows us to interpret vectors from the dual of the sum-of-squares cone as rigorous nonnegativity certificates of a WSOS polynomial. Whereas conventional WSOS certificates are alternative representations of the polynomials they certify, dual certificates are distinct from the certified polynomials; moreover, each dual certificate certifies a full-dimensional convex cone of WSOS polynomials. For a theoretical application, we give a short new proof of Powers' theorems on the existence of rational WSOS certificates of positive polynomials. For a computational application, we show that exact WSOS certificates can be constructed from numerically computed dual certificates at little additional cost, without any rounding or projection steps applied to the numerical certificates. We also present an algorithm for computing the optimal WSOS lower bound of a given polynomial along with a rational dual certificate, with a polynomial-time computational cost per iteration and linear rate of convergence.\\[-1em]
\end{abstract}
\alittlelessspace
\begin{keywords}
	polynomial optimization, nonnegativity certificates, sums-of-squares, non-symmetric conic optimization
\end{keywords}
\alittlelessspace
\begin{AMS}
	90C23, 14Q30, 90C51, 49M29, 90C25
\end{AMS}
\alittlelessspace

\section{Introduction}
Deciding whether a polynomial is nonnegative on an (often compact) semialgebraic set and the closely related problem of computing the (approximate) minimum value of a polynomial are fundamental problems of computational algebraic geometry and theoretical computer science, with many applications from discrete geometry and algorithmic theorem proving to the design and analysis of dynamical systems such as power networks, to name a few. This problem is well-known to be decidable \cite{Tarski1951, Renegar1992abc} but strongly NP-hard. The perhaps most studied, and arguably practically most successful, computational approach to it has been to certify the nonnegativity of the polynomial by writing it as a (weighted) sum of squared polynomials---a technique known as \emph{sum-of-squares decomposition}. A variety of results from real algebraic geometry such as Putinar's \emph{Positivstellensatz} \cite{Putinar1993} guarantee that every polynomial that is strictly positive over a compact semialgebraic set has such a representation.

Lower bounds on the global minima of polynomials and weighted sum-of-squares (WSOS) decompositions are usually computed numerically, using semidefinite programming (e.g., \cite{sostools, gloptipoly, Lasserre2001}) or non-symmetric cone optimization \cite{PappYildiz2019}, which is sufficient in many of the practical applications mentioned above. However, in many contexts, such as in computational algebraic geometry, system verification, and automated theorem proving, it is required that the computed bounds be certified rigorously, in exact arithmetic. \revision{Such rational certificates are not always guaranteed to exist. In particular, rational polynomials on the boundary of the sums-of-squares cone may not have a rational sums-of-squares decomposition \cite{Scheiderer2016}, \cite{NaldiSinn2021}. On the other hand, polynomials in the interior of the sums-of-squares cone \textit{do} have rational decompositions, a result due to Powers \cite{Powers2011}, of which we provide a new elementary proof in Section \ref{sec:rational-existence}.}

Computing rational WSOS decompositions for polynomials with rational coefficients is a challenging problem even in the univariate case \cite{MagronSafeyElDinSchweighofer2019}. Symbolic methods such as those that rely on quantifier elimination or root isolation are exponential in the \revision{number of variables} of the input polynomial, and in the univariate case have been found to be less efficient than more specialized methods. The optimal value of the semidefinite program is an algebraic number, but the study of the algebraic degree of the positive semidefinite cone \cite{NieRanestadSturmfels2010} suggests that one cannot hope for easily computable and verifiable certificates from taking a purely symbolic computational approach to the semidefinite programming problems that come from sums-of-squares. Therefore, a number of authors have proposed hybrid methods that ``round'' or ``project'' efficiently computable but inexact numerical sum-of-squares certificates to rigorous rational ones \cite{PeyrlParrilo2008, MagronSafeyElDin2018, DostertDeLaatMoustrou2020}; see also \cite{KaltofenLiYangZhi2008,KaltofenLiYangZhi2012,BrakeHauensteinLiddell2016}. \revision{(Certifying a WSOS lower bound of a polynomial, and the related goal of certifying a nonnegativity of a polynomial using rational certificates, is also discussed in \cite{MagronSafeyElDin2021}, and software developed to attain these goals includes \cite{MagronSafeyElDin2018-2} and, \cite{RAGLib}, amongst others.)}

Our contribution is twofold. In \mbox{Section \ref{sec:dualcertificates}}, we propose a new framework for certifying that a polynomial is WSOS using \emph{dual certificates}. The approach relies on convex programming duality and allows the efficient construction of rational WSOS decompositions from suitable rational vectors from the dual cone. In contrast to conventional WSOS certificates, which can be viewed as different representations of the polynomial whose nonnegativity they certify, dual certificates are distinct from the certified polynomials themselves. Moreover, every polynomial in the interior of the WSOS cone has a full-dimensional cone of dual certificates, which makes it particularly easy to identify one with an efficient numerical method. \revision{We also show that every rational polynomial in the interior of the WSOS cone has a rational dual certificate. This gives short new proofs to a number of known results about the existence of sum-of-rational-squares decompositions, including theorems of Powers \cite{Powers2011}.}

In \mbox{Section \ref{sec:algorithm}}, we discuss various algorithmic applications of dual certificates. We propose an efficient algorithm, \mbox{Algorithm \ref{alg:Newton}}, for computing and certifying rational WSOS lower bounds for polynomials over a compact semialgebraic set using dual certificates. The algorithm can be implemented as an entirely numerical method that nevertheless produces exact rational WSOS decompositions certifying rational lower bounds. 
The algorithm provides, in each iteration, a certifiable WSOS bound with a dual certificate that can be converted (in polynomial time) to an explicit rational WSOS decomposition without any additional rounding or projection of the numerical solutions. The sequence of WSOS bounds converges to the optimal WSOS bound at a linear rate. In \mbox{Section \ref{sec:univariate}}, we deduce explicit bounds on the number of iterations of Algorithm \ref{alg:Newton} in the univariate case. \revision{Section 5 includes additional examples demonstrating the efficacy of the method and the quality of the lower bounds obtained using Algorithm \ref{alg:Newton} in standard polynomial optimization benchmark problems. Specifically, we demonstrate that in some cases, the best certifiable bound using our purely numerical algorithm (implemented using double-precision floating point arithmetic) is indistinguishable from the true minimum in double-precision arithmetic.}

\setlength\abovedisplayskip{4pt}
\setlength\belowdisplayskip{4pt}

\subsection{Preliminaries}
In the rest of this section we introduce some notation and briefly review some convex optimization and interior-point theory that we rely on throughout the paper.

\subsubsection{Weighted SOS polynomials and positive semidefinite matrices} Recall that a convex set $K\subseteq\R^n$ is called a \emph{convex cone} if for every $\vx\in K$ and $\lambda \geq 0$ scalar, the vector $\lambda\vx$ also belongs to $K$. A convex cone is \emph{proper} if it is closed, \emph{full-dimensional} (meaning $\operatorname{span}(K)=\R^n$), and \emph{pointed} (that is, it does not contain a line). We shall denote the interior of a proper cone $K$ by $K^\circ$.

\paragraph{Sum-of-squares (SOS) polynomials} Let $\mathcal{V}_{n,2d}$ denote the cone of $n$-variate polynomials of degree $2d$. We say that a polynomial $p \in \mathcal{V}_{n,2d}$ is \emph{sum-of-squares} (SOS) if there exist polynomials $q_1,\dots,q_k \in \mathcal{V}_{n,d}$ such that $p = \sum_{i=1}^kq_i^2$. Define $\Sigma_{n,2d}$ to be the cone of $n$-variate SOS polynomials of degree $2d$. The cone $\Sigma_{n,2d}\subset \mathcal{V}_{n,2d} \equiv \R^{\binom{n+2d}{n}}$ is a proper cone for every $n$ and $d$.

\paragraph{Weighted sum-of-squares} More generally, let $\vw = (w_1,\dots,w_m)$ be some given nonzero polynomials and let $\mathbf{d} = (d_1,\dots,d_m)$ be a nonnegative integer vector. We denote by $\cV_{n,2\mathbf{d}}^\vw$ the space of polynomials $p$ for which there exist $r_1 \in \mathcal{V}_{n,2d_1}, \dots, r_m \in \cV_{n,2d_m}$ such that $p = \sum_{i=1}^m w_ir_i$. A polynomial $p \in\mathcal{V}_{n,2\mathbf{d}}^\vw$ is said to be \emph{weighted sum-of-squares} (WSOS) if there exist $\sigma_1 \in \Sigma_{n,2d_1}, \dots, \sigma_m \in \Sigma_{n,2d_m}$ such that $p = \sum_{i=1}^mw_i\sigma_i$. It is customary to assume that $w_1=1$, that is, the ordinary ``unweighted'' sum-of-squares polynomials are also included in the WSOS cones. Let $\Sigma_{n,2\mathbf{d}}^\vw$ denote the set of WSOS polynomials in $\cV_{n,2\mathbf{d}}^\vw$. \revision{By definition, $\Sigma_{n,2\mathbf{d}}^\vw \subset \mathcal{V}_{n,2\mathbf{d}}^\vw$ is a full-dimensional convex cone. Additionally, under mild conditions, the cone $\Sigma_{n,2\mathbf{d}}^\vw$ is closed and pointed; for example, it is sufficient that the set
\begin{equation}\label{eq:Swdef}
S_\vw \defeq \{\vx\in\R^n\,|\,w_i(\vx)\geq 0,\,i=1,\dots,m\}
\end{equation}
is a unisolvent point set for the space $\mathcal{V}_{n,2\mathbf{d}}^\vw$ \cite[Prop.~6.1]{PappYildiz2019}. (A set of points $S\subseteq\rr^n$ is \emph{unisolvent} for a space of polynomials $\mathcal{V}$ if every polynomial in $\mathcal{V}$ is uniquely determined by its function values at $S$.) In particular, this implies that both $\Sigma_{n,2\vd}^\vw$ and its dual cone have a non-empty interior.}

\paragraph{WSOS polynomials and positive semidefinite matrices}
We will denote the set of $n\times n$ real symmetric matrices by $\bS^n$, and the cone of positive semidefinite $n\times n$ real symmetric matrices by $\bS^n_+$. When the dimension is clear from the context, we use the common shorthands $\vA\succcurlyeq 0$ to denote that the matrix $\vA$ is positive semidefinite and $\vA\succ 0$ to denote that the matrix $\vA$ is positive definite. We will routinely identify polynomials with their coefficient vectors in a fixed basis of $\cV_{n,2\vd}^\vw$. Thus, $\cV_{n,2\vd}^\vw$ and $\left(\cV_{n,2\vd}^\vw\right)^*$ are identified with $\rr^U$, where $U = \dim\left(\cV_{n,2\vd}^\vw\right)$.

The following well-known theorem (rooted in the works of Shor, Lasserre, Parrilo, and Nesterov; here reproduced in the notation of the latter) illustrates the connection between $\Sigma_{n,2\mathbf{d}}^\vw$ and the cone of positive semidefinite matrices.
\begin{proposition}[\protect{\cite[Thm.~17.6]{Nesterov2000}}]\label{thm:Nesterov} Fix an ordered basis $\vq = (q_1,\dots,q_U)$ of $\mathcal{V}^\vw_{n,2\mathbf{d}}$ and an ordered basis $\vp_{i} = (p_{i,1},\dots,p_{i,L_i})$ of $\mathcal{V}_{n,d_i}$ for $i = 1,\dots,m$. Let $\Lambda_i: \cV_{n,2\vd}^\vw \left(\equiv \rr^U\right) \to \mathbb{S}^{L_i}$ be the unique linear mapping satisfying $\Lambda_{i}(\mathbf{q}) = w_i\mathbf{p}_i\mathbf{p}_i^T$, and let $\Lambda_i^*$ denote its adjoint. Then $\mathbf{s} \in \Sigma_{n,2\mathbf{d}}^\vw$ if and only if there exist matrices \mbox{$\mathbf{S}_1\succcurlyeq \mathbf{0}, \dots, \mathbf{S}_m \succcurlyeq \mathbf{0}$} satisfying 
\begin{equation}\label{eq:Nesterov-Lambdastar}
\vs = \sum_{i=1}^m\Lambda_i^*(\mathbf{S}_i).
\end{equation}
Additionally, the dual cone of $\Sigma_{n,2\mathbf{d}}^\vw$ admits the characterization 
\begin{equation}\label{eq:Nesterov-Lambda}
\left(\Sigma^\vw_{n,2\mathbf{d}}\right)^* = \left\{\vx \in \cV_{n,2\vd}^\vw \left(\equiv \rr^U\right) \ | \ \Lambda_i(\vx) \succcurlyeq \mathbf{0} ~~ \forall i=1,\dots,m\right\}.
\end{equation}
\end{proposition}
The proof of Proposition~\ref{thm:Nesterov} is constructive: given matrices $\vS_i\in\bS^{L_i}_+\,(i=1,\dots,m)$, one may explicitly construct a (weighted) sum-of-squares decomposition of the polynomial $\vs$. Thus, the collection of matrices $(\vS_1,\dots,\vS_m)$ itself can be interpreted as a WSOS certificate of the polynomial $\vs$.

\revision{
\begin{example}\label{ex:Lambda}
Throughout this example, matrices and vectors are indexed from $0$; $x_i$ denotes the $i$th component of the vector $\vx$, and $S_{ij}$ denotes the $(i,j)$-th entry of the matrix $\vS$.
\begin{enumerate}
\item If both $\vp$ and $\vq$ from Proposition \ref{thm:Nesterov} are the standard monomial bases of univariate polynomials of degree $d$ and $2d$, respectively, then $\Lambda: \rr^{2d+1} \to \mathbb{S}^{d+1}$ maps a vector in $\rr^{2d+1}$ to its corresponding Hankel matrix in $\mathbb{S}^{d+1}$, and $\Lambda^*$ maps a symmetric matrix $\vS$ in $\mathbb{S}^{d+1}$ to vectors in $\rr^{2d+1}$ by summing along its antidiagonals. For example, if $d=2$ (and so $L = d+1 = 3$ and $U = 2d+1 = 5$), then $\Lambda$ and its adjoint are given by the equations
\[
\Lambda(\vx) = \begin{pmatrix} 
x_0 & x_1 & x_2  \\ x_1 & x_2 & x_3 \\ x_2 & x_3 & x_4
\end{pmatrix} 
\]
and 
\[
\Lambda^*(\vS) = (S_{00}, 2S_{01}, 2S_{02}+S_{11}, 2S_{12}, S_{22}).
\]

\item If both $\vp$ and $\vq$ from Proposition \ref{thm:Nesterov} are univariate Chebyshev polynomials of degree $d$ and $2d$, respectively, then letting $T_i$ correspond to the $i$th Chebyshev polynomial and using the identity $T_iT_j = \frac{1}{2}\left(T_{i+j} + T_{|i - j|} \right)$, we deduce that the $\Lambda$ operator in this setting satisfies $\Lambda(\vx)_{ij} = \frac{x_{i+j} + x_{|i - j|}}{2}$. 
For example, for $d=2$ and $\vx \in \rr^5$, 
\[
\Lambda(\vx) = \begin{pmatrix} 
x_0 & x_1 & x_2  \\ 
x_1 & \frac{x_0 + x_2}{2} & \frac{x_1 + x_3}{2}  \\ x_2 & \frac{x_1 + x_3}{2} & \frac{x_0 + x_4}{2} \end{pmatrix}
\]
is a Hankel-plus-Toeplitz matrix and
\[
\Lambda^*(\vS) = \left(S_{00}+\frac{S_{11}+S_{22}}{2}, 2S_{01}+S_{12}, 2S_{02}+\frac{S_{11}}{2}, S_{12}, \frac{S_{22}}{2}\right).
\]
\item 
If both $\vp$ and $\vq$ from Proposition \ref{thm:Nesterov} are the standard univariate monomial bases, and we use weights $1$ and $1 - z^2$, then $\Lambda_1(\vx)$ is a standard Hankel matrix and $\Lambda_2(\vx)$ is a shifted Hankel matrix (sometimes also referred to as \emph{localizing matrices} \cite{Laurent2009}). For example, if $d = 2$, then $U=5, L_1=3, L_2=2$, and the operator $\Lambda \defeq \Lambda_1 \oplus \Lambda_2\colon \rr^5 \to \bS^3 \oplus \bS^2$ is given by 
\[
\Lambda(x_0,x_1,x_2,x_3,x_4) =  
\begin{pmatrix}
x_0 & x_1 & x_2 \\
x_1 & x_2 & x_3 \\
x_2 & x_3 & x_4
\end{pmatrix} \oplus
\begin{pmatrix}
x_0 - x_2 & x_1 - x_3 \\
x_1 - x_3 & x_2 - x_4
\end{pmatrix};
\]
see, for example, \cite[Sec.~II.2]{KarlinStudden1966}.
The adjoint operator is given by
\[\Lambda^*(\vS^1\oplus\vS^2) = \left(S^1_{00}+S^2_{00},
2 S^1_{01}+2 S^2_{01},
2 S^1_{02}+S^1_{11}-S^2_{00}+S^2_{11},
2 S^1_{12}-2 S^2_{01},S^1_{22}-S^2_{11}\right).\]
\end{enumerate} 
\end{example}
}

To lighten the notation, throughout the rest of the paper we assume that the weight polynomials $\vw=(w_1,\dots,w_m)$ and the degrees $\vd=(d_1,\dots,d_m)$ are fixed, and denote the cone $\Sigma_{n,2\vd}^\vw$ by $\Sigma$ and the space of polynomials $\mathcal{V}_{n,2\mathbf{d}}^\vw$ by $\mathcal{V}$. Additionally, we denote by $\Lambda$ the $\R^U\to \bS^{L_1} \oplus \cdots \oplus \bS^{L_m}$ linear map $\Lambda_1(\cdot)\oplus \cdots \oplus\Lambda_m(\cdot)$ from Proposition~\ref{thm:Nesterov}. With this notation, the condition \eqref{eq:Nesterov-Lambdastar} can be written as $\vs=\Lambda^*(\vS)$ for some positive semidefinite (block diagonal) matrix $\vS\in\bS^{L_1} \oplus \cdots \oplus \bS^{L_m}$. Similarly, Eq.~\eqref{eq:Nesterov-Lambda} simplifies to
\begin{equation}\label{eq:SigmaStar}
\Sigma^* = \{\vx\in\R^U\,|\,\Lambda(\vx)\succcurlyeq \vzero\}.
\end{equation}
The interior of this cone is simply
\begin{equation}\label{eq:intSigmaStar}
\Ssc = \{\vx\in\R^U\,|\,\Lambda(\vx)\succ \vzero\}.
\end{equation}

\subsubsection{Barrier functions and local norms in convex cones}
The analysis of the dual certificates introduced in Section~\ref{sec:dualcertificates} relies heavily on the theory of barrier functions for convex cones. In this section, we give a brief overview of the parts of this theory and some additional notation that will be needed throughout the rest of the paper.

It is convenient to identify the spaces $\cV$ and $\cV^*$ with $\R^U$ ($U=\dim(\cV)$), equipped with  the standard inner product, $\langle \vx, \vy \rangle = \vx^\T\vy$ and the induced Euclidean norm $\|\cdot\|$.

Let $\Lambda: \R^U \to \bS^L$ be the unique linear mapping specified in Proposition~\ref{thm:Nesterov} above, and let $\Lambda^*$ denote its adjoint. Central to our theory is the \emph{barrier function} $f: \Ssc \to \R$ defined by
\begin{equation}\label{eq:def:f}
    f(\vx) \defeq -\ln(\det(\Lambda(\vx)).
\end{equation}
Note that by Eq.~\eqref{eq:intSigmaStar}, $f$ is indeed defined on its domain. The function $f$ is twice continuously differentiable; we denote by $g(\vx)$ its gradient at $\vx$ and by $H(\vx)$ its Hessian at $\vx$. Since $f$ is strictly convex on its domain, $H(\vx)\succ \vzero$ for all $\vx \in \Ssc$. Consequently, we can also associate with each $\vx\in\Ssc$ the \emph{local inner product} $\langle\cdot,\cdot\rangle_\vx : \cV^* \times \cV^* \to \rr$ defined as $\langle\vy,\vz\rangle_\vx \defeq \vy^\T H(\vx) \vz$ and the \emph{local norm} $\|\cdot\|_\vx$ induced by this local inner product. Thus, $\|\vy\|_\vx = \|H(\vx)^{1/2}\vy\|$. We define the local (open) ball centered at $\vx$ with radius $r$ by $B_{\vx}(\vx,r) \defeq \{\vy\in\cV^*\,|\, \|\vy - \vx\|_{\vx} < r\}$.
Analogously, we define the \emph{dual local inner product} $\langle \cdot,\cdot \rangle_\vx^*: \cV\times\cV\to\R$ by  $\langle \vs, \vt \rangle_\vx^* \defeq \vs^\T H(\vx)^{-1}\vt$. The induced \emph{dual local norm} $\|\cdot\|_{\vx}^*$ satisfies the identity $\|\vt\|_{\vx}^* = \|H(\vx)^{-1/2}\vt\|$.

We remark that the function in \eqref{eq:def:f} falls into the broader category of \emph{logarithmically homogeneous self-concordant barriers} (or LHSCBs for short), which are expounded upon in the classic texts \cite{NesterovNemirovskii1994} and \cite{Renegar2001}. Throughout, we will invoke several useful results concerning LHSCBs for the function \eqref{eq:def:f}; these are enumerated in the following lemma:

\begin{lemma}\label{thm:f-properties} 
Using the notation introduced in this section, the following hold for every $\vx\in\Ssc$:
\begin{enumerate}
    \item We have $B_{\vx}(\vx,1) \subset \left(\Sigma^*\right)^\circ$, and for all $\vu\in B_{\vx}(\vx,1)$ and $\vv \neq 0$, one has
    \begin{equation}\label{eq:self-concordance}
    1 - \|\vu - \vx\|_\vx \leq \frac{\|\vv\|_\vu}{\|\vv\|_\vx}\leq (1 - \|\vu - \vx\|_\vx)^{-1}.
    \end{equation}
    \item The gradient $g$ of $f$ can be computed as
    \begin{equation}\label{eq:g}
    g(\vx) = -\Lambda^*(\Lambda(\vx)^{-1}),
    \end{equation}
    and the Hessian $H(\vx)$ is the linear operator satisfying
    \begin{equation}\label{eq:H}
    H(\vx) \vv= \Lambda^*\!\left(\Lambda(\vx)^{-1}\Lambda(\vv)\Lambda(\vx)^{-1}\right) \;\; \text{ for every } \vv \in \rr^U.
    \end{equation}
    \item The function $f$ is \emph{logarithmically homogeneous}; that is, it has the following properties:
    \begin{equation}\label{eq:log-homogeneity}
    g(\alpha \vx) = {\alpha}^{-1} g(\vx) \;\text{ and }\; H(\alpha \vx) = {\alpha}^{-2}H(\vx) \;\; \text{ for every } \alpha > 0,
    \end{equation}
    furthermore
    \begin{equation}\label{eq:gH-identities}
    H(\vx)\vx = -g(\vx)\quad\text{and}\quad \|g(\vx)\|_{\vx}^* = \|\vx\|_{\vx} = \sqrt{\langle -g(\vx),\vx\rangle} = \sqrt{\nu},
    \end{equation}
    where $\nu = \sum_{i=1}^m L_i$ is the \emph{barrier parameter} of $f$.
    \item \label{item:bijection} The gradient map $g:\Ssc\to \R^U$ defines a bijection between $\left(\Sigma^*\right)^\circ$ and ${\Sigma}^\circ$, In particular, for every $\vs \in {\Sigma}^\circ$ there exists a unique $\vx \in \left(\Sigma^*\right)^\circ$ satisfying $\vs = -g(\vx)$.
    \item If $\|\vu-\vx\|_\vx < 1$, then
    \begin{equation}\label{eq:lemma5} 
        \|g(\vu) - g(\vx)\|_{\vx}^* \leq \frac{\|\vu - \vx\|_\vx}{1 - \|\vu - \vx\|_\vx}.
    \end{equation}
    \item If $\|g(\vu) - g(\vx)\|_{\vx}^* < 1$, then 
    \begin{equation}\label{eq:revlemma5}
        \|\vu - \vx\|_{\vx}\leq \frac{\|g(\vu) - g(\vx)\|_{\vx}^*}{1 - \|g(\vu) - g(\vx)\|_\vx^*}.
    \end{equation}
\end{enumerate}
\end{lemma} 

\begin{proof}\phantom{ }\\
\be 
\item This is Renegar's definition of self-concordance applied to the function $f$, which is a composition of an affine function and a well-known self-concordant function, and is thus self-concordant; see \cite[Sec.~2.2.1 and Thm.~2.2.7]{Renegar2001}. 
\item Straightforward calculation.
\item Straightforward calculation using the identities \eqref{eq:g} and \eqref{eq:H}. We remark that these identities hold for all LHSCBs \cite[Thm.~2.3.9]{Renegar2001}.
\item See \cite[Sec.~3.3]{Renegar2001}.
\item See \cite[Lemma~5]{PappYildiz2017corrigendum}.
\item This is an application of the previous claim to the conjugate barrier function of $f$. \qedhere
\ee 
\end{proof} 


\section{Dual certificates}\label{sec:dualcertificates}
We begin this section by introducing our central object, the cone of dual certificates corresponding to a WSOS polynomial (\mbox{Definition \ref{def:dual-certificate}}), and by showing in \mbox{Theorem \ref{thm:main}} how we can use dual certificates to construct an explicit (weighted) sum-of-squares decomposition of WSOS polynomials in closed form. We continue using the notation introduced in the previous section, and let $\Sigma$ denote a general WSOS cone $\Sigma_{n,2\vd}^\vw$ with non-empty interior and $H$ denote the Hessian of the barrier function $f$ defined in \eqref{eq:def:f}.

\begin{definition}\label{def:dual-certificate}
Let $\vs\in\Sigma$. We say that the vector $\vx\in\Ssc$ is a \emph{dual certificate of $\vs$}, or simply that \emph{$\vx$ certifies $\vs$}, if $H(\vx)^{-1}\vs \in \Sigma^*$. We denote by
\[\cC(\vs) \defeq \{\vx\in\Ssc\,|\,H(\vx)^{-1}\vs \in \Sigma^*\}\]
the set of dual certificates of $\vs$. 
Conversely, for every \revision{$\vx\in\Ssc$}, we denote by
\[\cP(\vx) \defeq \{\vs\in\Sigma\,|\,H(\vx)^{-1}\vs \in \Sigma^*\}\]
the set of polynomials certified by the dual vector $\vx$.
\end{definition}

\revision{\begin{example} To keep this example as simple as possible, we consider unweighted univariate sum-of-squares polynomials represented in the monomial basis. Consider the univariate polynomial $z \mapsto 1 + z^2$ corresponding to the coefficient vector $\vs = (1,0,1) \in \Sigma$, the (unweighted) sums-of-squares cone of polynomials with degree at most 2. We can characterize $\cC(\vs)$ as follows. By definition, a vector $\vx=(x_0,x_1,x_2)$ belongs to the cone of dual certificates $\cC(\vs)$ if and only if $\Lambda(H(\vx)^{-1}\vs) \succcurlyeq \mathbf{0}$ and $\vx \in \Ssc$. The inverse Hessian at $\vx$ is 
\[
H(\vx)^{-1} = \begin{pmatrix}x_0^2 & x_0x_1 & x_1^2 \\ x_0x_1 & \frac{1}{2}(x_1^2 + x_0x_2) & x_1x_2 \\ x_1^2 & x_1x_2 & x_2^2 \end{pmatrix},
\]
so $H(\vx)^{-1}\vs = (x_0^2 + x_1^2, x_0x_1 + x_1x_2, x_1^2 + x_2^2)$. Then, we have
\[
\Lambda(H(\vx)^{-1}\vs)=  \begin{pmatrix} x_0^2 + x_1^2 & x_0x_1 + x_1x_2 \\ 
x_0x_1 + x_1x_2 & x_1^2 + x_2^2 
\end{pmatrix} = \Lambda(\vx)^2,
\]  so $\Lambda(H(\vx)^{-1}\vs)$ is automatically positive semidefinite. Hence $\vx \in \cC(\vs)$ for every $\vx \in \Ssc$.
\end{example}
}

The following theorem justifies the terminology introduced in Definition \ref{def:dual-certificate}. Through \mbox{Eq. \eqref{eq:Sdef}} below, we can construct a WSOS certificate $\vS$ for the polynomial $\vs$ in the spirit of \mbox{Proposition \ref{thm:Nesterov}} by an efficiently-computable closed-form formula, and thus we may interpret the dual vector $\vx\in\cC(\vs)$ itself as a certificate of the polynomial $\vs$.
\begin{theorem}\label{thm:main}
Let $\vx\in\Ssc$ be arbitrary. Then the matrix $\vS=\vS(\vx,\vs)$ defined by
\begin{equation}\label{eq:Sdef}
\vS(\vx,\vs) \defeq \Lambda(\vx)^{-1}\Lambda\!\left(H(\vx)^{-1}\vs\right)\Lambda(\vx)^{-1}
\end{equation}
satisfies $\Lambda^*(\vS) = \vs$. Moreover, $\vx$ is a dual certificate for $\vs\in\Sigma$ if and only if $\vS\succcurlyeq 0$.
\end{theorem}
\begin{proof}
The first statement can be shown by applying the Hessian formula from \mbox{Lemma \ref{thm:f-properties}}:
\[
    \Lambda^*(\mathbf{S}) \overset{\eqref{eq:Sdef}}{=} \Lambda^*\left(\Lambda(\vx)^{-1}\Lambda\!\left(H(\vx)^{-1}\vs\right)\Lambda(\vx)^{-1}\right) \overset{\eqref{eq:H}}{=} H(\vx)H(\vx)^{-1}\vs = \vs,
\]

For the second statement, note that $\vS \succcurlyeq \vzero$ if and only if $\Lambda\!\left(H(\vx)^{-1}\vs\right) \succcurlyeq \vzero$, which is equivalent to $\vx\in\cC(\vs)$ by the definition of $\cC(\vs)$ and the characterization \eqref{eq:SigmaStar} of $\Sigma^*$.
\end{proof}

Recall from Lemma~\ref{thm:f-properties} (claim \ref{item:bijection}) that for every $\vs\in\Sigma^\circ$ there exists a unique $\vx \in \Ssc$ satisfying $\vs=-g(\vx)$. This vector is a dual certificate of $\vs$, since
\[H(\vx)^{-1}\vs = -H(\vx)^{-1}g(\vx) \overset{\eqref{eq:gH-identities}}{=} \vx \in \Ssc.\]
Thus, every polynomial in the interior of the WSOS cone $\Sigma$ has a dual certificate.
\begin{definition}
When $-g(\vx) = \vs\, (\in\!\Sigma^\circ)$, we say that $\vx$ is the \emph{gradient certificate of $\vs$}.
\end{definition}

It is immediate from the definition that if $\vx$ is a dual certificate of $\vs$, then so is every positive multiple of $\vx$. \revision{Analogously, if $\vx$ is a dual certificate of $\vs$, then $\vx$ is also a dual certificate of every positive multiple of $\vs$.} (One may also confirm directly that the matrix $\vS$ constructed in \eqref{eq:Sdef} is invariant to a positive scaling of $\vx$.)
Also note that when $\vx$ is the gradient certificate of $\vs=-g(\vx)$, then $\vS(\vx,\vs)$ is positive definite. Since $\vS$ is continuous on $\Ssc\times\Sc$, all vectors in some ($\vs$-dependent) neighborhood of $\vx$ are dual certificates of $\vs$, as they also give rise to a positive semidefinite $\vS(\vx,\vs)$. Conversely, the gradient certificate of $\vs$ is also a dual certificate of every polynomial in some ($\vx$-dependent) neighborhood of $\vs$. Our next theorem is a quantitative version of this observation. (Recall that $\nu$ denotes the barrier parameter introduced in \mbox{Eq. \eqref{eq:gH-identities}}.)

\begin{theorem}\label{thm:sufficient-cone} Suppose $\vt\in\Sigma^\circ$ and let $\vx\in\Ssc$ be any vector that satisfies the inequality
\begin{equation}\label{eq:sufficient-cone}
    \vt^\T\left(\vx\vx^\T - (\nu - 1)H(\vx)^{-1}\right)\vt \geq 0. 
\end{equation}
Then $\vx\in\cC(\vt)$, and equivalently, $\vt\in\cP(\vx)$. In particular, if $\vs=-g(\vx)$ for some $\vx\in\Ssc$, then $\vx$ is a dual certificate for every polynomial $\vt$ satisfying $\|\vt-\vs\|_\vx^*\leq 1$.
\end{theorem} 

\begin{proof} We start with the second claim. From the definitions of the local norm and the dual local norm, we have
\begin{equation}\label{eq:distance-identities}
    \|\vt - \vs\|_{\vx}^* = 
    \|H(\vx)^{-1/2}(\vt - \vs)\| = 
    \|H(\vx)^{1/2}(\vx - H(\vx)^{-1}\vt)\| = 
    \|\vx - H(\vx)^{-1}\vt\|_\vx .
\end{equation}
Thus, $\|\vt - \vs\|_{\vx}^* \leq 1$ is equivalent to $H(\vx)^{-1}\vt \in \overline{B_\vx(\vx,1)}$. Since $B_\vx(\vx,1) \subseteq \left(\Sigma^*\right)^\circ$ from the first claim of Lemma \ref{thm:f-properties}, $\overline{B_\vx(\vx,1)} \subseteq \Sigma^*$, and $\vx \in \cC(\vt)$ by definition.

The first claim of the Lemma is the ``conic version'' of the second claim. To prove it, suppose that the inequality in \eqref{eq:sufficient-cone} holds. Then the univariate quadratic polynomial  
\[z \mapsto (1 - \nu)z^2 + \left( 2 \langle \vt,\mathbf{x}\rangle \right)z - \langle \vt, H(\mathbf{x})^{-1}\vt\rangle\] 
has a nonnegative discriminant, therefore it has a \revision{real} root $\delta$. Moreover, since $(1-\nu) < 0$ and $\langle \vt, H(\vx)^{-1}\vt \rangle > 0$, it follows that 
$\delta > 0$. 
Using the identities in Eq.~\eqref{eq:gH-identities}, we have
\begin{align*}
    0 &\leq  (1 - \nu)\delta^2 + \left( 2 \langle \vt,\mathbf{x}\rangle \right)\delta - \langle \vt, H(\mathbf{x})^{-1}\vt\rangle\\
    &= \delta^2 \left(1 - \langle g(\mathbf{x}),H(\mathbf{x})^{-1} g(\mathbf{x})\rangle\right) - \delta\left( 2 \langle \vt,H(\mathbf{x})^{-1}g(\mathbf{x})\rangle \right) - \langle \vt, H(\mathbf{x})^{-1}\vt\rangle\\
    &=\delta^2 - \langle\vt + \delta g(\mathbf{x}), H(\mathbf{x})^{-1}(\vt + \delta g(\mathbf{x})) \rangle \\
    &=\delta^2 - \|H(\mathbf{x})^{-1/2}(\vt + \delta g(\mathbf{x}))\|^2.
\end{align*} 
We conclude that $\|H(\mathbf{x})^{-1/2}(\vt + \delta g(\mathbf{x}))\| < \delta$ for some $\delta > 0.$  Then using Lemma~\ref{thm:f-properties} again, we have 
\begin{align*}
    1 &\geq \frac{1}{\delta} \|H(\mathbf{x})^{-1/2}(\vt + \delta g(\mathbf{x}))\| \\
    &\overset{\eqref{eq:gH-identities}}{=} \left\|\delta H(\vx)^{1/2} \left(\delta^{-2}H(\vx)^{-1}\vt - \delta^{-1}\vx\right)\right\| \\
    &\overset{\eqref{eq:log-homogeneity}}{=}\left\|H\left(\delta^{-1}\vx\right)^{1/2}\left(H\left(\delta^{-1}\vx\right)^{-1}\vt - \delta^{-1}\vx\right)\right\|, 
\end{align*}
so by the identities \eqref{eq:distance-identities} and the first part of our proof, $\vt$ is certified WSOS by $\frac{1}{\delta}\vx$. Since all positive multiples of $\vx$ certify $\vt$, and $\delta$ is positive, it follows that $\vx$ certifies $\vt$. \end{proof}

\begin{corollary}\label{thm:x-ycert}
Suppose that $\vx\in\Sigma^*$, $\vs=-g(\vx)$, and that $\vy$ is a vector that satisfies the inequality $\|\vx - \vy\|_\vx < \frac{1}{2}$. Then $\vy\in\Sigma^*$, and $\vx$ certifies $\vt=-g(\vy)$.
\end{corollary} 

\begin{proof}If $\|\vx - \vy\|_\vx < \frac{1}{2}$, Lemma \ref{thm:f-properties} ensures that $\vy\in\Sigma^*$, and we also have
\[
    \|\vs - \vt\|_{\vx}^* =\|g(\mathbf{x}) - g(\mathbf{y})\|_{\vx}^* \overset{\eqref{eq:lemma5}}{\leq} \frac{\|\mathbf{x} - \mathbf{y}\|_{\vx}}{1 - \|\mathbf{x}-\mathbf{y}\|_\mathbf{x}}  < 1.
 \]
Then by \revision{Theorem \ref{thm:sufficient-cone}}, $\vx$ certifies $\vt$.  \end{proof}

\begin{example} \label{ex:simple} Consider the univariate polynomial $t$ given by $t(z) =1 - z + z^2 + z^3 - z^4$. To show that $t$ is nonnegative on the interval $[-1,1]$, it suffices to show that the coefficient vector $\vt = (1 , -1, 1, 1, -1)$ is a member of $\Sigma^{\vw}_{1,2\vd}$, with weights $\vw(z) = (1,1-z^2)$ and degree vector $\vd = (2,1)$. For this example, we represent all polynomials in the monomial basis. \revision{In this setting, the $\Lambda$ and $\Lambda^*$ operators are precisely those detailed in Example \ref{ex:Lambda} part (3). }


Consider the vector $\vx=(5, 0, 5/2, 0, 15/8)$. This vector is the gradient certificate of the constant one polynomial, since simple arithmetic yields that $-g(\vx)=\Lambda^*(\Lambda(\vx)^{-1})=(1,0,0,0,0)$. The same certificate also certifies the nonnegativity of the polynomial $t$ above. To confirm this, we compute $H(\vx)^{-1}$: 
\[
H(\vx)^{-1} =\frac{5}{384} \begin{pmatrix}
     384 & 0 & 192 & 0 & 144 \\ 
     0 & 240 & 0 & 180 & 0 \\
     192 & 0 & 176 & 0 & 152 \\
     0 & 180 & 0 & 165 & 0  \\
     144 & 0 & 152 & 0 & 149 
     \end{pmatrix}, 
\] and by Theorem \ref{thm:main}, it is sufficient to verify that
\[ \frac{128}{5}\Lambda\left(H(\vx)^{-1}\vt\right) = \begin{pmatrix}
 144 & -20 & 72\\
 -20 & 72 & -5 \\
 72 & -5 & 49
\end{pmatrix} \oplus
\begin{pmatrix}
72 & -15 \\
-15 & 23 \\
\end{pmatrix}\succcurlyeq \vzero. \]

With some additional work, we can also compute from $\vx$ rational matrices $\vS_1$ and $\vS_2$ \revision{to certify the nonnegativity of the polynomial using Proposition \ref{thm:Nesterov}}: plugging our dual certificate into the formula \eqref{eq:Sdef}, we obtain
\[
\vS_1 = \frac{1}{40}\begin{pmatrix}
 22 & -5 & -26 \\
 -5 & 18 & 5 \\
 -26 & 5 & 52
 \end{pmatrix}
 \quad\text{and}\quad
\vS_2 = \frac{1}{40}\begin{pmatrix}
 18 & -15 \\
 -15 & 92 \\
\end{pmatrix}.
\]
These matrices, in turn, can be factored using the $LDL^\T$ form of Cholesky decomposition to compute an explicit rational sum-of-squares representation of $t$:
\begin{align*}
t(z) =& 
\frac{11}{20} \left(-\frac{13 z^2}{11}-\frac{5 z}{22}+1\right)^2 +
\frac{371}{880} \left(z-\frac{20 z^2}{371}\right)^2 +
\frac{3937 z^4}{7420} + \\
& + \left(1-z^2\right) \left(
\frac{9}{20} \left(1-\frac{5 z}{6}\right)^2 +
\frac{159 z^2}{80}
\right).
\end{align*}
\end{example}

\revision{
\subsection{Rational nonnegativity certificates}
We now turn to an important theoretical application of dual certificates and show that every \emph{rational polynomial} (that is, every polynomial with rational coefficients) in the interior of $\Sigma$ has a rational dual certificate, and show that this implies that every rational polynomial in $\Sc$ is a sum of \emph{rational} squares. (We cannot hope this to be true on the boundary; this follows from Scheiderer's seminal result \cite{Scheiderer2016}, although facial reduction techniques could yield rational certificates in some cases \cite{NaldiSinn2021}.) As an immediate consequence, we get new short proofs of two theorems of Powers \cite{Powers2011} that positive rational polynomials over compact sets, under the assumptions of Schm{\"u}dgen's (resp., Putinar's) Positivstellensatz,  have rational sum-of-squares certificates.

We also show using similar arguments that these rational sum-of-squares certificates can be computed from dual certificates using numerical algorithms.

\subsubsection{Existence of rational WSOS certificates for rational polynomials}\label{sec:rational-existence}

\begin{theorem}\label{thm:rational-intSigma}
Every $\vs \in \Sigma^\circ \cap \qq^U$ has a rational dual certificate $\vx \in \Ssc \cap \qq^U$ and a sum-of-rational-squares decomposition.
\end{theorem}
\begin{proof}
Let $\vs \in \Sigma^\circ$, and suppose $\vy \in \Ssc$ is the gradient certificate of $\vs$. As $\cC(\vs)$ is full-dimensional for every $\vs \in \Sigma^\circ$ by Theorem \ref{thm:sufficient-cone}, and as $H(\vy)$ is positive definite for every $\vy \in \Ssc$, the ball centered at $\vy$ with radius $1/2$ in the local norm is a full-dimensional ellipsoid contained in $\Ssc$. As a result, there exists a rational vector $\vx \in \qq^U \cap \Ssc$ such that $\|\vx - \vy\|_{\vy} \leq 1/2$, which by Corollary \ref{thm:x-ycert} guarantees that $\vx$ is a rational dual certificate for $\vs$. Finally, by Theorem \ref{thm:main}, we can derive a conventional (primal) weighted sum-of-\textit{rational}-squares decomposition for $\vs$: the Gram matrix $\vS(\vx,\vs)$ defined in \eqref{eq:Sdef} is automatically rational, and from there an explicit rational WSOS decomposition can be computed in rational arithmetic via the $LDL^\T$ decomposition of this Gram matrix.
\end{proof}

Recall that according to Putinar's celebrated Positivstellensatz \cite{Putinar1993}, every (strictly) positive polynomial over a compact semialgebraic set \eqref{eq:Swdef} belongs to a WSOS cone $\Sigma = \Sigma_{n,2\vd}^\vw$ for a sufficiently large $\vd$, if the quadratic module associated with the polynomials $w_1,\dots,w_m$ is Archimedean. It follows that (under the same assumption) every positive polynomial belongs to the \emph{interior} of $\Sigma_{n,2\vd}^\vw$ as well for every sufficiently large $\vd$. (Every positive polynomial $\vs$ lies in the interior of a simplex whose vertices are positive polynomials, and by Putinar's Positivstellensatz each vertex of this simplex belongs to $\Sigma_{n,2\vd}^\vw$ for every large enough $\vd$; our $\vs$ then belongs to the interior of this WSOS cone.) Invoking Theorem \ref{thm:rational-intSigma}, we have the following ``rational Putinar's Posivstellensatz'' for rational polynomials:


\begin{corollary}[\protect{\cite[Thm.~7]{Powers2011}}] Let $\vs \in \qq^U$ be the coefficient vector of a polynomial $s(\cdot)$ that is strictly positive on the compact semialgebraic set $S_\vw$ defined in \eqref{eq:Swdef}, and suppose that the quadratic module associated with $\{w_1,\dots,w_m\}$ is Archimedean and each $w_i$ is rational. Then $s$ has weighted-sum-of-rational-squares decomposition using the weights $w_1,\dots,w_m$.
\end{corollary} 

In an analogous fashion, we can show that as long as $S_\vw$ is compact (without the Archimedean assumption), every rational polynomial positive over $S_\vw$ has a rational Schm{\"u}dgen nonnegativity certificate as was also shown by Powers; see \cite[Thm.~5]{Powers2011}.

}


\subsubsection{Rigorous certificates from numerical methods} \label{sec:rational-certificates}

Theorem \ref{thm:sufficient-cone} and Corollary \ref{thm:x-ycert} have important consequences for both numerical (finite-precision), exact-arithmetic, and hybrid algorithms for computing sum-of-squares certificates.

The fact that every polynomial $\vs\in\Sigma^\circ$ has a full-dimensional cone of certificates $\cC(\vs)$ means that exact dual certificates can in principle be computed by purely numerical, inexact algorithms. As long as $\vs$ is a rational polynomial that is sufficiently in the interior of $\Sigma$ that a numerical method (implemented in floating point arithmetic) can identify \emph{any} point $\vx\in \cC(\vs)$, the same argument as in the proof of Theorem \ref{thm:rational-intSigma} shows that a rational weighted-sum-of-squares certificate can be readily computed.

We can also take this argument one step further and apply it to certifying sum-of-squares lower bounds. Consider, for example, a hypothetical algorithm that aims to compute the gradient certificate $\vy$ of some polynomial $t-\gamma$ with coefficient vector $\vt-\gamma\vone\in\Sigma^\circ$ to certify $t(\vz)\geq \gamma\,\forall\vz\in S_\vw$, but computes instead only an approximation $\vx\approx \vy$ in finite-precision arithmetic. (Here, $\vone$ denotes the coefficient vector of the constant one polynomial; $S_\vw$ is the semialgebraic set \eqref{eq:Swdef} over which we wish to bound the polynomial $t$.) As long as the inherent errors of the finite-precision computation are small enough to ensure $\|\vx-\vy\|_\vx \leq 1/2$, \mbox{Corollary \ref{thm:x-ycert}} guarantees that $\vx$ is a certificate of nonnegativity for $t-\gamma$. Since floating-point numbers are, by definition, rational, every sufficiently accurate numerical solution of $-g(\vx) = \vt-\gamma\vone$ is automatically a rational dual certificate of the nonnegativity of $t-\gamma$. Additionally, as long as the coefficient vector $\vt$ and the lower bound $\gamma$ are also rational, any such numerical dual certificate $\vx$ can be directly converted to an exact rational primal certificate (Gram matrix) $\vS \succcurlyeq \vzero$ satisfying $\Lambda^*(\vS) = \vt-\gamma\vone$ via the formula of Eq.~\eqref{eq:Sdef}.

This property sets dual certificates apart from conventional certificates: a numerical solution to the semidefinite programming (feasibility) problem
\[ \text{find an } \vS\succcurlyeq \vzero \text{ satisfying } \Lambda^*(\vS) = \vt-\gamma\vone \]
will generally satisfy the equality constraints $\Lambda^*(\vS) = \vt-\gamma\vone$ only within some numerical tolerance, thus $\vS$ will not be a rigorous certificate, even if we can guarantee (by the appropriate choice of optimization algorithm) that at least the cone constraint $\vS\succcurlyeq\vzero$ is always satisfied. Hence, additional post-processing (a rounding or projection step, such as those in the hybrid methods of \cite{PeyrlParrilo2008} and \cite{MagronWang2020}) is needed. In contrast, any dual certificate $\vx$ from the full-dimensional cone $\cC(\vt-\gamma\vone)$ is a rational certificate (that can be turned into an explicit rational WSOS decomposition).

In Section \ref{sec:Newton}, we present an efficient algorithm (Algorithm \ref{alg:Newton}) to compute certifiable rational lower bounds with matching dual certificates that can be implemented as an entirely numerical method using these ideas.

\subsection{Complexity considerations}\label{sec:complexity}

Depending on the choice of the $\Lambda$ operator (that is, in essence, the choice of bases $\vp$ and $\vq$ in the construction of the semidefinite representation of $\Sigma$ following \mbox{Proposition \ref{thm:Nesterov}}), the computation of $\vS(\vx,\vs)$ can be made efficient, even polynomial-time in the bit model. Suppose that for a given rational $\vx\in\Ssc$, the matrices $\Lambda(\vx)$ and $H(\vx)$ are rational and can be computed efficiently. Then for any $\vs\in\R^U$, the computation of $\vS(\vx,\vs)$ amounts to (1) computing a rational Cholesky ($LDL^\T$) factorization of $\Lambda(\vx)$ and $H(\vx)$ (which are positive definite by definition); (2) computing the vector $\vv=H(\vx)^{-1}\vs$ using the Cholesky factors of $H(\vx)$ computed in the previous step; and (3) computing $\Lambda(\vv)$ and then $\vS(\vx,\vs)$ using the Cholesky factors of $\Lambda(\vx)$. 
Therefore, computing $\vS(\vx,\vs)$ is efficient as long as $\Lambda(\cdot)$ and $H(\cdot)$ can be computed efficiently. 

For any reasonable choice and representation of $\Lambda$, the computation of $\Lambda(\cdot)$ and $\Lambda^*(\cdot)$ are efficient, as they are linear operators, typically explicitly represented in matrix form with rational entries. Studying the same question in the context of numerical methods for SOS optimization, the authors in \cite[Sec.~6]{PappYildiz2019} showed that when polynomials are represented as Lagrange interpolants, the Hessian $H(\vx)$ can be computed with $\mathcal{O}(m\max_i\{L_i\}U^2)$ arithmetic operations. One can also argue directly from the identity \eqref{eq:H}, that (since $\Lambda$ and $\Lambda^*$ are efficiently computable) the Hessian can be computed efficiently; the bottleneck once again is the inversion or factorization of $\Lambda(\vx)$. We note the monomial and Chebyshev polynomial bases as two additional important special cases (both in the univariate and multivariate setting): in these cases, $\Lambda(\vx)$ is a low displacement-rank matrix. For example, when the polynomials are univariate, each block of $\Lambda$ is a Hankel (or Hankel-plus-Toeplitz) matrix if using the monomial (or Chebyshev) basis. Therefore the inversion of $\Lambda$ and the computation of $H$ can be handled using discrete Fourier transforms or the ``superfast'' (nearly-linear-time) algorithms of Pan and others \cite{Pan2001}.

\revision{
\subsection{Relation to prior work}\label{sec:priorwork}
(Weighted) sum-of-squares certificates are commonly associated with, and computed using, semidefinite optimization---an approach that goes back to Nesterov \cite{Nesterov2000}, Parrilo \cite{ParriloThesis2000}, and Lasserre \cite{Lasserre2001}. It is pertinent to put our work in the context of Lasserre's, as both make extensive use of the dual cone $\Sc$ and thus have many superficial similarities.

In our notation, the semidefinite optimization approach can be summarized as follows: to find the best WSOS lower bound for a polynomial $t$, whose coefficient vector is denoted by $\vt$, we need to solve the semidefinite optimization problem
\begin{equation}\label{eq:SOS-SDP-primal} \sup \{c\,|\, \Lambda^*(\vS) = \vt-c\vone,\, \vS\succcurlyeq 0\}.
\end{equation}
The WSOS certificate itself is the \emph{Gram matrix} $\vS$, whose factorization yields an explicit representation of $t-c$. Lasserre's seminal observation is that the WSOS lower bound can also be characterized as the optimal value of the dual semidefinite optimization problem (which can also be derived from moment theory), written as
\begin{equation}\label{eq:SOS-SDP-dual}
\inf \{\vt^\T\vy\,|\, \vone^\T\vy = 1, \Lambda(\vy) \succcurlyeq 0\}.
\end{equation}
In fact, an immediate consequence of weak duality is that every feasible solution of \eqref{eq:SOS-SDP-dual} yields a WSOS-certifiable lower bound on $t$. Additionally, under standard regularity conditions (such as the existence of a Slater point in \eqref{eq:SOS-SDP-dual}) we have strong duality, with attainment in the primal problem \eqref{eq:SOS-SDP-primal}, meaning that the optimal value of \eqref{eq:SOS-SDP-dual} is the best WSOS-certifiable lower bound \cite{Lasserre2001}. 

Our work extends this theory. Although the above arguments show that every $\vy\in\Sigma^*$ with $\vone^\T\vy=1$ yields a WSOS lower bound, these dual vectors cannot be turned into explicit (``primal'') WSOS certificates---a dual optimal (or feasible) solution $\vy$ from \eqref{eq:SOS-SDP-dual} does not translate to an optimal (or feasible) solution $\vS$ for any particular $c$ in \eqref{eq:SOS-SDP-primal}. One interpretation of Theorem \ref{thm:sufficient-cone} is that it identifies a full-dimensional subset of solutions of \eqref{eq:SOS-SDP-dual} which, through the definition \eqref{eq:Sdef}, can be turned into a primal certificate $\vS$ via a simple closed-form formula. This allows us to circumvent solving the semidefinite optimization problem \eqref{eq:SOS-SDP-primal} altogether. This is a potentially huge gain, as the dimension of $\vS$ in \eqref{eq:SOS-SDP-primal} is considerably larger than $\dim(\Sigma)$. We demonstrate how dual certificates can be used in designing efficient algorithms in Section \ref{sec:Newton}.
}

\section{Computing rigorously certified lower bounds with dual certificates}\label{sec:algorithm}

With our theoretical infrastructure and notation in place, we now turn to the question of computing certified lower bounds and dual certificates for these bounds. In \mbox{Section~\ref{sec:EveryPolyHasaBound}} we show that under the condition that the constant one polynomial is in the interior of our WSOS cone, every polynomial has a dual certifiable lower bound. (We argue that this is a mild, essentially without loss of generality, condition in \mbox{Section \ref{sec:assumptions}}.) We also show that after a suitable preprocessing (required only once for every WSOS cone), such a certified bound can be computed by a closed-form formula for any polynomial.

In \mbox{Section \ref{sec:c-update}} we discuss efficient algorithms to compute the best lower bound that a given certificate certifies for a given polynomial and show that using dual certificates, inexact numerical certificates (that come, for example, from numerical sum-of-squares optimization approaches) can be turned into rigorous rational certificates with minimal additional effort.

We then combine these ideas with the observations made in \mbox{Section \ref{sec:rational-certificates}} and present a new algorithm (Algorithm \ref{alg:Newton}) for approximating the best WSOS lower bound for a given polynomial with arbitrary accuracy in Section \ref{sec:Newton}. The algorithm returns both a rational lower bound approximating the optimal WSOS lower bound and a rational certificate certifying the bound. We also show that Algorithm \ref{alg:Newton} is linearly convergent to the optimal bound. In \mbox{Section \ref{sec:C-bounds}}, we detail how to compute a bound on the linear rate of convergence of \mbox{Algorithm \ref{alg:Newton}}. This in turn makes it possible to compute WSOS lower bounds that are certifiably within a prescribed $\varepsilon$ from the optimal bound. 

Throughout this section, and the rest of the paper, the boldface vector $\vone$ represents the constant one polynomial (or, precisely, its coefficient vector) in the WSOS cone $\Sigma\, (=\Sigma_{n,2\vd}^\vw)$, in the space of polynomials $\cV (=\cV_{n,2\vd}^\vw)$.

\subsection{Universal dual certificates} \label{sec:EveryPolyHasaBound}
Suppose that $\vone\in\Sc$. Then $\vone$ has a gradient certificate $\vxo$, and as we have seen in Theorem \ref{thm:sufficient-cone}, $\vone\in\cP(\vxo)^\circ$, that is, $\vxo$ certifies an entire full-dimensional cone of polynomials with $\vone$ in its interior. Conversely, an entire cone of certificates, with $\vxo$ in its interior, certifies $\vone$. Our next observation is that each of these certificates also certifies \emph{some} WSOS lower bound for \emph{every} polynomial:
\revision{
\begin{lemma}\label{thm:EveryPolyHasaBound}
Let $\vx\in\Ssc$ be any certificate for which $\vone\in\cP(\vx)^\circ$ and $r\in(0,1/2]$. Then for every polynomial $\vt \in \cV$, the inclusion $\vx\in\cC(\vt+c\vone)$ holds for every sufficiently large scalar $c$. Specifically, if $\vx \in \Ssc$ satisfies $\|-g(\vx) - \vone\|_{\vx}^* \leq \frac{r}{r+1}$, then the inclusion $\vx \in \cC(\vt+c\vone)$ holds for every
\begin{equation}\label{eq:c0-bound}
c \geq \frac{\|\vt\|^*_{\vx}}{\frac{r}{r+1} - \|-g(\vx) - \vone\|_{\vx}^*}.
\end{equation}
In this case, letting $\vy_c$ denote the gradient certificate of $\vt+c\vone$, the inequality
\begin{equation*}
\|c^{-1}\vx- \vy_c\|_{c^{-1}\vxo} \leq r
\end{equation*}
also holds.
\end{lemma}
\begin{proof}The first statement is immediate from the fact that $\cP(\vx)$ is a cone and the assumption that $\vone\in\cP(\vx)^\circ$: the dual vector $\vx$ certifies all small perturbations of $\vone$, including every polynomial of the form $(c^{-1}\vt + \vone)$, and thus also $\vt+c\vone$, for every sufficiently large $c$. We prove the second and third statements in detail.

Using the definitions of the local dual norm and logarithmic homogeneity \eqref{eq:log-homogeneity} from Lemma~\ref{thm:f-properties}, we have
\begin{equation}\label{eq:thm-initial-12}
  \|(\vt + c   \mathbf{1}) - c \mathbf{1}\|_{c^{-1}\vx}^* \overset{(\text{by def.})}{=} \|H(c^{-1}\vx)^{-1/2}\vt\| \overset{(\ref{thm:f-properties})}{=} c^{-1}\|H(\vx)^{-1/2}\vt\| \overset{(\text{by def.})}{=} c^{-1}\|\vt\|_{\vx}^* 
\end{equation}
Similarly, 
\begin{equation}\label{eq:thm-initial-13}
  \|c\vone + cg(\vx)\|_{c^{-1}\vx}^* \overset{\text{by def.}}{=} c\|H(c^{-1}\vx)^{-1/2}(\vone + g(\vx))\| \overset{(\ref{thm:f-properties})}{=} \|H(\vx)^{-1/2}(\vone+g(\vx))\| \overset{\text{by def.}}{=} \|\vone + g(\vx)\|_{\vx}. 
\end{equation}
Thus, we have
\begin{equation}\label{eq:initial-3-eq}
    \begin{split}
        \|(\vt + c\vone) + cg(\vx)\|^*_{c^{-1}\vx} &\leq \|(\vt + c\vone) - c\vone\|_{c^{-1}\vx}^* + \|c\vone + cg(\vx)\|_{c^{-1}\vx}^* \\
        &\overset{\eqref{eq:thm-initial-12},\eqref{eq:thm-initial-13}}{=} c^{-1}\|\vt\|_{\vx}^* +\|\vone +g(\vx)\|_{\vx}^*  \\
        &\overset{\eqref{eq:c0-bound}}{\leq} \frac{r}{r+1} - \|-g(\vx) - \vone\|_{\vx}^* + \|-g(\vx) - \vone\|_{\vx}^*  \\
        &= \frac{r}{r+1}.
    \end{split}
\end{equation}
Using logarithmic homogeneity again, we see that $c^{-1}\vx$ is the gradient certificate for $-cg(\vx)$. Therefore, invoking Theorem~\ref{thm:sufficient-cone}, we deduce from the inequality \eqref{eq:initial-3-eq} that $c^{-1}\vx$ is a dual certificate for $\vt + c\vone$. 
Moreover, via the inequality \eqref{eq:revlemma5} in Lemma~\ref{thm:f-properties}, we conclude that
\[
\|c^{-1}\vx - \vy_c\|_{c^{-1}\vx} \overset{\eqref{eq:revlemma5}}{\leq} \frac{\|\vt\|^*_{c^{-1}\vx}}{1-\|\vt\|^*_{c^{-1}\vx}} \overset{\eqref{eq:initial-3-eq}}{\leq} r,
\]
as claimed.
\end{proof} 

}

We emphasize that the certificate $\vxo$ (or any $\vx$ with $\vone\in\cP(\vx)^\circ$) in Lemma \ref{thm:EveryPolyHasaBound} only needs to be computed once for any particular WSOS cone $\Sigma_{n,2\vd}^\vw$. Once $\vxo$ (and the corresponding $H(\vxo)^{-1}$) are computed, a certifiable lower bound and a corresponding certificate can be computed in closed form for every polynomial $\vt \in \cV$, with minimal effort.

When the weight polynomials $\vw$ are sufficiently simple, the gradient certificate of $\vone$ may even be easily expressible in closed form, as in the following example.

\begin{example}\label{ex:Chebyshev}
Consider the cone of nonnegative univariate polynomials of degree $2d$ over the interval $[-1,1]$, which is well known to be the same as the WSOS cone $\Sigma_{n,2\vd}^\vw$ with $n=1$, $m=2$, degree vector $\vd = (d,d-1)$, and weight polynomials $\vw(z) = (1,1-z^2)$ \cite{BrickmanSteinberg1962}. Furthermore, suppose that all polynomials are represented in the basis of Chebyshev polynomials of the first kind, that is, both of the ordered bases $\vp$ and $\vq$ in \mbox{Proposition \ref{thm:Nesterov}} that determine the operator $\Lambda$ are Chebyshev basis polynomials. Then both diagonal blocks of $\Lambda$ are Hankel-plus-Toeplitz matrices (similar to \mbox{Example \ref{ex:Lambda}}), and the gradient certificate of $\vone=(1,0,\dots,0) \in \R^{2d+2}$ is simply the vector
\[
\vxo = (2d+1,0,\dots,0).
\]
This can be proven by a direct calculation verifying the equality $-g(\vxo) = \Lambda^*(\Lambda(\vxo)^{-1}) = \vone$.
The Hessian at this certificate is the diagonal matrix
\begin{equation}\label{eq:Hx1}
H(\vxo) = \frac{1}{2d+1}\operatorname{diag}\left(1, \frac{4d}{2d+1}, \frac{4d-2}{2d+1}, \dots, \frac{2}{2d+1}\right).
\end{equation}
Analogous results can be derived for polynomials of odd degree using $\vd = (d,d)$, and weight polynomials $\vw(z) = (1-z,1+z)$.

\end{example}


\subsection{Optimal and near-optimal lower bounds from a given dual certificate}\label{sec:c-update}

Suppose we have found a dual certificate $\vx$ that certifies the nonnegativity of the polynomial $\vt-c\vone$. What is the \emph{best} lower bound certified by the same certificate? By definition, the answer is the solution of the one-dimensional optimization problem
\[
c_\text{max} \defeq \max \left\{\gamma\in\R\,|\,\vt-\gamma\vone \in \cP(\vx)\right\}.
\]
As discussed in Section~\ref{sec:dualcertificates}, if the inverse Hessian $H(\vx)^{-1}$ (or the Cholesky or $LDL^\T$ factorization of $H(\vx)$) is already computed, then membership in $\cP(\vx)$ is easy to test by verifying the positive semidefiniteness of $\Lambda(H(\vx)^{-1}(\vt-\gamma\vone))$. Therefore, an arbitrarily close lower approximation of $c_\text{max}$ can be found efficiently, in time proportional to the logarithm of the approximation error, by binary search on the optimal $\gamma$. (An initial lower bound on $c_\text{max}$ is the currently certified lower bound $c$ assumed to be part of the input; an upper bound on $c_\text{max}$ can be computed, e.g., by evaluating the polynomial $\vt$ at any point in its domain.)

The repeated matrix factorization makes the algorithm outlined above too expensive to use as a subroutine. A weaker bound can be computed \emph{in closed form} using \mbox{Theorem \ref{thm:sufficient-cone}}: if 
\[
c_\text{max}^\prime \defeq \max \left\{\gamma \in\R\,\middle|\, (\vt-\gamma\vone)^\T\left(\vx\vx^\T - (\nu - 1)H(\vx)^{-1}\right)(\vt-\gamma\vone) \geq 0 \right\},
\]
then $\vt-c_\text{max}^\prime \in \cP(\vx)$. For a given certificate $\vx$, if the inverse Hessian $H(\vx)^{-1}$ (or the Cholesky or $LDL^\T$ factorization of $H(\vx)$) is already computed, then solving this optimization problem amounts to finding the roots of a univariate quadratic function.



\begin{example}\label{ex:simple-cmax}
Continuing with Example \ref{ex:simple} (with $\vt = (1, -1, 1, 1, -1)$, weights $\vw(z) = (1, 1 - z^2)$), we compute $c_\text{max}$ and $c_\text{max}^{\prime}$ for $t$ using the certificate $\vx = (5, 0, 5/2, 0, 15/8)$. For comparison, the minimum of the polynomial is $\frac{1}{512} \left(619-51 \sqrt{17}\right)\approx 0.798$.

To compute $c_{\text{max}}$, we compute the largest $\gamma$ such that $\Lambda(H(\vx)^{-1}(\vt - \gamma\vone)) = \Lambda_1(H(\vx)^{-1}(\vt - \gamma\vone)) \oplus \Lambda_2(H(\vx)^{-1}(\vt - \gamma\vone))$ is positive semidefinite but not positive definite. We compute the characteristic polynomials of the $\gamma$-parametrized matrices, as $\Lambda(H(\vx)^{-1}(\vt - \gamma\vone))$ is on the boundary of the PSD cone when the constant term of the characteristic polynomial vanishes. 
The constant term of the characteristic polynomial of $\Lambda_1(H(\vx)^{-1}(t - \gamma\vone))$, itself a polynomial in $\gamma$, has smallest real root at $\gamma = \frac{1}{64}\left(67 - 5\sqrt{17}\right)$. 
Meanwhile, the constant term of the  characteristic polynomial of $\Lambda_2(H(\vx)^{-1}(t - \gamma\vone))$ has smallest real root at $\gamma = \frac{1}{32} \left(41 - 5\sqrt{10}\right)$. 
We conclude that $c_{\text{max}} = \frac{1}{64}\left(67 - 5\sqrt{17}\right) \approx 0.724$. 

To compute $c_{\text{max}}^{\prime}$, we expand and reduce
\[
(\vt-\gamma\vone)^\T\left(\vx\vx^\T - (\nu - 1)H(\vx)^{-1}\right)(\vt-\gamma\vone) = \frac{205}{64} - \frac{45\gamma}{4} + 5\gamma^2.
\]
Computing the roots of this quadratic, we conclude that $c_{\text{max}}^{\prime} = \frac{1}{8}\left(9 - 2\sqrt{10}\right) \approx 0.334$. 

\end{example}

\subsection{Computing optimal WSOS bounds}\label{sec:Newton}
We now present an iterative method to compute \emph{the best WSOS lower bound} for a given polynomial $\vt$ along with a certificate for that bound. The pseudocode of the algorithm is shown in \mbox{Algorithm \ref{alg:Newton}}. After a high-level description of the method, we show that it converges linearly to the optimal WSOS bound below (\mbox{Theorem \ref{thm:linear-convergence}}).

\afterpage{
\setlength{\textfloatsep}{0pt}
\begin{algorithm}[t]
  \DontPrintSemicolon
  \SetKwInOut{Input}{input}
  \SetKwInOut{Output}{outputs}
  \SetKwInOut{Params}{parameters}
  \Input{A polynomial $\vt$; a tolerance $\varepsilon>0$.}
  \Params{An oracle for computing the barrier Hessian $H$ for $\Sigma$; a radius $r\in(0,1/4]$; \revision{a certificate $\vx$ satisfying $\|-g(\vx) - \vone\|_{\vx}^* \leq \frac{r}{r+1}$}.}
  \Output{A lower bound $c$ on the optimal WSOS lower bound $c^*$ satisfying $c^*-c\leq \varepsilon$; a dual vector $\vx\in\Ssc$ certifying the nonnegativity of $\vt-c\vone$.}

  \medskip

  \revision{Compute $c_0 := -\left(\frac{r}{r+1} - \|-g(\vx) - \vone\|_{\vx}^*\right)^{-1}\|\vt\|^*_{\vx}.$ Set $c := c_0$ and $\mathbf{x} := -\frac{1}{c_0}\vx$.}
  \label{line:init}
  
  \Repeat{$\Delta c \leq \rho_rC\varepsilon$ \do}{
    Set $\vx := 2\vx - H(\vx)^{-1}(\vt -c\vone )$. \label{line:x-update}
    
    Find the largest real number $c_+$ such that 
    \[
    \|\vx - H(\vx)^{-1}(\vt - c_+\vone)\|_\vx \leq \frac{r}{r+1}.
    \]\label{line:c-update}

    Set $\Delta c := c_+ - c$.
    Set $c := c_+$.
    }\label{line:stopping}
    \Return{$c$ and $\vx$.}
    \caption{Compute the best WSOS lower bound and a dual certificate}\label{alg:Newton}
\end{algorithm}
} 

Previously, in Lemma \ref{thm:EveryPolyHasaBound}, we showed that for a sufficiently large $c$, $\vt + c\vone$ can be certified by $c^{-1} \vx$ for every $\vx$ in a suitable neighborhood of the gradient certificate of $\vone$; this result justifies the initialization of the algorithm in Line \ref{line:init}. In order to increase the lower bound, the algorithm iterates two steps: certificate updates (Line \ref{line:x-update}) and bound updates (Line \ref{line:c-update}). The bound updates are similar to the $c_\text{max}^{\prime}$ bound in \mbox{Section \ref{sec:c-update}}; we will precisely justify this step in Lemma~\ref{thm:constant-update}. The certificate updates are motivated as follows: since each bound update attempts to push $c$ towards the best bound certifiable by $\vx$, the certificate $\vx$ sits near the boundary of $\cC(\vt-c\vone)$ after each bound update. To allow for a sufficient additional increase of the bound in the subsequent iteration, the certificate $\vx$ is updated to be closer to the gradient certificate $\vy$ of the current $\vt-c\vone$. This certificate $\vy$ would be prohibitively expensive to compute in each iteration; instead, the update step in Line \ref{line:x-update} can be interpreted as a single Newton step from $\vx$ towards the solution of the nonlinear system $-g(\vy) = \vt-c\vone$.

\begin{example}\label{ex:simple-alg}
We continue with the setup of Examples \ref{ex:simple} and \ref{ex:simple-cmax}: we consider the univariate polynomial whose coefficient vector in the monomial basis is $\vt = (1, -1, 1, 1, -1)$, defined over the interval $[-1,1]$ represented by the weights $\vw(z) = (1,1-z^2)$. \mbox{Algorithm \ref{alg:Newton}} with $r=1/4$, with inputs $\vt$ and tolerance $\varepsilon = 10^{-7}$ in double-precision floating point arithmetic outputs the bound $c\approx 0.798284319$ 
 and a certificate vector $\vx$. Note that the exact minimum of $t$ is $\frac{1}{512}(619-51\sqrt{17})\approx 0.798284401$.

A plot of the difference between the current certified lower bound $c$ and the minimum $c^*$ in each iteration is shown in Figure \ref{fig:simple-alg}, illustrating the linear convergence of \mbox{Algorithm \ref{alg:Newton}} for this polynomial. The exact rational representation of the floating point bound is
\[
c = 2^{-53} \cdot 7190305926654593,
\]
and the rational vector certifying the nonnegativity of $\vt-c\vone$ is
\[
 \vx = 2^{-33} \begin{pmatrix}173493184462864992\\ 67729650226350000\\  -120611300436615200\\ -161900156381728960\\ -5796381308580693\end{pmatrix}.
\]
Note that no rounding or projection steps are needed to compute a rigorous certificate. In the analysis of the algorithm below (Lemma \ref{thm:constant-update}) we shall see that if the algorithm was implemented in exact arithmetic, we would have $\|\vx-\vy\|_\vx \leq r = 1/4$ in each iteration\revision{, where $\vy$ is the gradient certificate of $\vt - c\vone$}. Working with finite precision, the iterates may fail to satisfy this inequality; however, as long as the numerical errors are sufficiently small to ensure the considerably weaker inequality $\|\vx-\vy\|_\vx \leq 1/2$, the computed numerical certificate $\vx$ is automatically a rational certificate for the computed SOS lower bound $c$ by Corollary \ref{thm:x-ycert}.
\end{example}
\setlength{\textfloatsep}{3ex}
\afterpage{
\begin{figure}[!t]
    \centering
    \includegraphics[scale=.30]{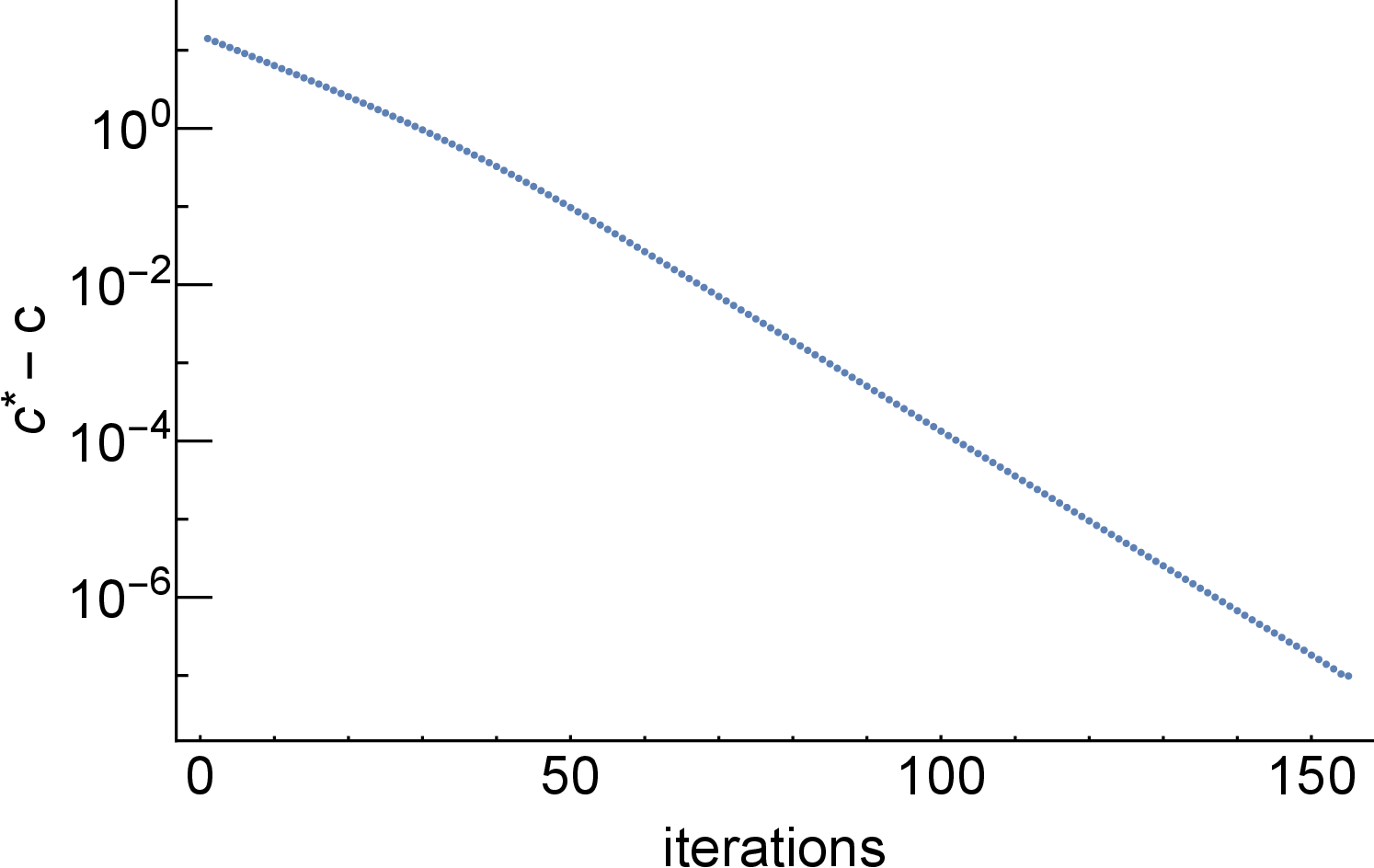}\\
    \caption{The convergence of the sequence of certified lower bounds computed by Algorithm \ref{alg:Newton} to the minimum of the polynomial studied in Examples \ref{ex:simple} and \ref{ex:simple-alg}, illustrating the linear convergence shown in Theorem \ref{thm:linear-convergence} below.}
    \label{fig:simple-alg}
\end{figure}
} 

\revision{Additional examples are discussed in Section \ref{sec:examples}.}

The computationally most expensive part in each iteration is having to compute (after each certificate update) a Cholesky factorization of the Hessian $H(\vx)$ (or the inverse Hessian $H(\vx)^{-1}$). With that available,  the bound update and the next certificate update are very efficient: by an argument analogous to the discussion on $c_{\max}^\prime$ in the previous section,  the bound update amounts to solving a univariate quadratic equation, and the certificate update is essentially a matrix-vector multiplication or two triangular solves. As discussed in Section~\ref{sec:complexity}, the computation and factorization of the Hessian is efficient for popular choices of polynomial bases.

We now turn to the analysis of the algorithm, deferring the discussion on the stopping criterion until later. To simplify the statements of the results, we will use the following notation throughout the rest of the section. We define $\vx_+ \defeq 2\vx - H(\vx)^{-1}(\vt -c\vone )$ to be the updated certificate in Line \ref{line:x-update} to help distinguish the certificates before and after the update. 
Finally, we let $\vy$ be the vector satisfying $-g(\vy) = \vt - c\vone$ and $\vy_+$ be the vector satisfying $-g(\vy_+) = \vt - c_+ \vone$.

In the next series of Lemmas we show that the bound update from $c$ to $c_+$ is well-defined, and is always an increase, by bounding the distance between $\vx$ and $\vy$ in each step of the iteration. We also establish that throughout the algorithm, the iterates satisfy $\|\vx-\vy\|_\vx \leq r$. (At the beginning of the first iteration this holds by Lemma \ref{thm:EveryPolyHasaBound}.) The first result, Lemma~\ref{thm:xtoxplus}, shows that $\vx_+$ is closer than $\vx$ to the gradient certificate of $\vt - c\vone$ in their respective local norms.

\begin{lemma}\label{thm:xtoxplus} Let $\vx_+$ and $\vy$ be defined as above, and assume that $\|\vx-\vy\|_\vx \leq r$ for some $r<\frac{1}{3}$. Then $\|\vx_+ - \vy\|_{\vx_+}  \leq \frac{r^2}{1 - 2r}$.
\end{lemma} 

\begin{proof} Recall that the update in Line \ref{line:x-update} of Algorithm \ref{alg:Newton} is a single (full) Newton step towards the solution of the nonlinear system $-g(\vy) = \vt - c\vone$. Equivalently, the update $\vx_+ - \vx$ is a Newton step toward the minimizer of the convex self-concordant function 
\[
f_c(\vx) \defeq (\vt - c\vone)^T\vx + f(\vx).
\]
Applying \cite[Thm.~2.2.3]{Renegar2001} to $f_c$, we have 
\[
\|\vx_+ - \vy\|_{\vx} \leq \frac{\|\vx - \vy\|_\vx^2}{1 - \|\vx - \vy\|_{\vx}} = \frac{r^2}{1 - r}.
\] Coupling this result with the definition of self-concordance (Eq.~\eqref{eq:self-concordance}), we have
\[
\|\vx_+ - \vy\|_{\vy} \leq \frac{\|\vx_+ - \vy\|_\vx}{1 - \|\vx - \vy\|_\vx} \leq \frac{\|\vx - \vy\|_\vx^2}{(1 - \|\vx - \vy\|_\vx)^2} \leq \frac{r^2}{(1-r)^2} < 1.
\]
We conclude that $\vx_+ \in B_{\vy}(\vy,1)$, and we can thus invoke the inequality \eqref{eq:self-concordance} for another change of norms to conclude that
\[\|\vx_+ - \vy\|_{\vx_+} \leq \frac{\|\vx_+- \vy\|_{\vy}}{1 - \|\vx_+ -\vy\|_{\vy}} \leq \frac{\frac{r^2}{(1-r)^2}}{1-\frac{r^2}{(1-r)^2}} = \frac{r^2}{(1-r)^2 - r^2} = \frac{r^2}{1 - 2r}.
\qedhere\]
\end{proof}

We remark that, while $\vx$ certifies $\vt - c\vone$ whenever $\|\vx - \vy\|_{\vx}  <\frac{1}{2}$, and each step of our proof is valid for all $0<r<\frac{1}{2}$, we can only have $\frac{r^2}{1-2r} \leq r$ whenever $0 < r < \frac{1}{3}$. Therefore, using Lemma~\ref{thm:xtoxplus}, we can guarantee that $\|\vx_+ - \vy\|_{\vx_+} \leq \|\vx - \vy\|_{\vx}$ when $\|\vx - \vy\|_{\vx} < \frac{1}{3}$. Below, we need to further limit $r$ to ensure that the bound update is an improvement.

\begin{lemma}\label{thm:constant-update}
Suppose that $\|\vx_+ - \vy\|_{\vx_+} \leq \frac{r^2}{1-2r}$ for some $0<r\leq\frac{1}{4}$. 
Then $c_+>c$ and $\|\vx_+ - \vy_+\|_{\vx_+} \leq r$.
\end{lemma}
\begin{proof}
We begin by showing that
\[\|\vx_+ - H(\vx)^{-1}(\vt - c\vone)\|_{\vx_+} < \frac{r}{r+1},\]
which implies that Step \ref{line:c-update} of the algorithm indeed increases the lower bound to $c_+>c$.

Suppose $\|\vx_+ - \vy\|_{\vx_+} \leq \frac{r^2}{1-2r}$.
Recall from Eq.~\eqref{eq:gH-identities} that $H(\vx_+)\vx_+ = -g(\vx_+)$. Using this identity and the definition of the local norm, we deduce that
\begin{equation}
\begin{split}\label{eq:gdiff-identities}
    \|-g(\vx_+) + g(\vy)\|_{\vx_+}^* &= \|H(\vx_+)^{-1/2}\left(H(\vx_+)\vx_+ - (\vt-c\vone)\right)\| \\
    &=\|H(\vx_+)^{1/2}\vx_+ - H(\vx_+)^{-1/2}(\vt -c\vone)\| \\
    &=\|\vx_+ - H(\vx_+)^{-1}(\vt-c\vone)\|_{\vx_+}.
\end{split}
\end{equation}
Using this in tandem with inequality \eqref{eq:lemma5} from Lemma \ref{thm:f-properties}, we have  
\begin{equation*}
\begin{split}
    \|\vx_+ - H(\vx)^{-1}(\vt - c\vone)\|_{\vx_+} &\overset{\eqref{eq:gdiff-identities}}{=} \|-g(\vx_+) + g(\vy)\|_{\vx_+}\\
    &\overset{\eqref{eq:lemma5}}{\leq} \frac{\|\vx_+ - \vy\|_{\vx_+}}{1 - \|\vx_+ - \vy\|_{\vx_+}} \leq \frac{\frac{r^2}{1-2r}}{1 - \frac{r^2}{1-2r}} < \frac{r}{r+1}
\end{split}
\end{equation*}
for every $r\leq\frac{1}{4}$, proving our first claim.

To see the second statement, we observe that 
\begin{equation}\label{eq:gdiff-identities2}
\|-g(\vx_+) + g(\vy_+)\|_{\vx_+}^* = \|\vx_+ - H(\vx_+)^{-1}(\vt-c_+\vone)\|_{\vx_+} = \frac{r}{r+1} < 1\end{equation}
by the definition of the bound update step in Line~\ref{line:c-update} and our discussion above. Now inequality~\eqref{eq:revlemma5} from Lemma~\ref{thm:f-properties} yields
\begin{equation*}
\|\vx_+ - \vy_+\|_{\vx_+} \leq \frac{\|-g(\vx_+) + g(\vy_+)\|_{\vx_+}^*}{1 - \|-g(\vx_+) + g(\vy_+)\|_{\vx_+}^*} \leq \frac{ r/( r+1)}{1 - ( r/( r+1))} =  r.\qedhere
\end{equation*}
\end{proof}

The next lemma uses Lemma~\ref{thm:xtoxplus} in showing that the improvement in the lower bound can be bounded from below by a constant times the local norm of $\vone$. 

\begin{lemma}\label{thm:m-m*}  Define
$\rho_r \defeq \frac{r(1-3r-2r^2)}{1-r-2r^2}$ 
Then at the end of each iteration of \mbox{Algorithm \ref{alg:Newton}}, $c_+ - c \geq \frac{\rho_r}{\|\vone\|_{\vy}^*}$, where $\vy$ is the gradient certificate of $\vt - c\vone$.
\end{lemma}
\begin{proof}
From the identities \eqref{eq:gdiff-identities2} and the definition of $c_+$ in Line \ref{line:c-update} of the algorithm, we have
\[
\frac{r}{r+1} = \|\vx_+ - H(\vx_+)^{-1}(\vt - c_+\vone)\|_{\vx_+} = \|-g(\vx_+) + g(\vy_+)\|_{\vx_+}^*.
\]
Upper bounding the right-hand side by the triangle inequality gives  
\begin{equation}\label{eq:thm:m-m*1}
\frac{r}{r+1} -\|-g(\vx_+) + g(\vy)\|_{\vx_+}^* \leq  \|-g(\vy_+) +g(\vy)\|_{\vx_+}^* = \|(c_+ - c)\vone\|_{\vx_+}^*.
\end{equation} 
Thus, to lower bound $(c_+ - c)$, it suffices to upper bound $\|-g(\vx_+) + g(\vy)\|_{\vx_+}^*$. 

From Lemma~\ref{thm:xtoxplus}, we know that  $\|\vx_+ - \vy\|_{\vx_+} \leq \frac{r^2}{1-2r}$. Using the inequality \eqref{eq:lemma5} in Lemma~\ref{thm:f-properties}, we have 
\begin{equation}\label{eq:thm:m-m*2}
\|-g(\vx_+) + g(\vy)\|^*_{\vx_+} \leq\frac{\|\vx_+ - \vy\|_{\vx_+}}{1 - \|\vx_+ - \vy\|_{\vx_+}} \leq  \frac{\frac{r^2}{1-2r}}{1 - \frac{r^2}{1-2r}}  = \frac{r^2}{1 - 2r - r^2}.
\end{equation}
Combining the inequalities in \eqref{eq:thm:m-m*1} and \eqref{eq:thm:m-m*2}, we have
\begin{equation*}
(c_+ - c)\|\vone\|_{\vx_+}^* \geq \frac{r}{r+1} - \frac{r^2}{1 - 2r - r^2}.
\end{equation*}
Finally, changing norms again with inequality \eqref{eq:self-concordance}, 
\begin{align*}
(c_+ - c)\|\vone\|_{\vy_+}^* &\geq (c_+ - c)\|\vone\|_{\vx_+}^* (1 - \|\vy-\vx_+\|_{\vx_+}^*)\\
& \geq \left(\frac{r}{r+1} - \frac{r^2}{1 - 2r - r^2}\right)\left(1-\frac{r^2}{1-2r}\right) = \rho_r.
\end{align*}
\end{proof} 

We remark that if $r$ is chosen so that $0 < r \leq \frac{1}{4}$, then $\rho_r > 0$, and, for example, $\rho_r > 2/21$ for $r=1/6$. Therefore in each iteration of the algorithm, the improvement of the bound can be bounded from below by a quantity proportional to $(\|\vone\|_{\vy}^*)^{-1}$, where $\vy$ is the current gradient certificate. 

Now, we turn our attention to the convergence of Algorithm \ref{alg:Newton}. When $\vone\in\Sc$, the optimal WSOS lower bound $c^*$ for a polynomial $\vt$ is the unique scalar $\gamma$ for which $\vt-\gamma\vone$ is on the boundary of $\Sigma$.
In Theorem \ref{thm:lincon}, we show that the norm $\|\vone\|_{\vy}^*$ can be related to the distance $(c^*-c)$ between the current bound and the optimal WSOS lower bound. We will then combine this result with Lemma~\ref{thm:m-m*} above to show that the algorithm converges linearly to the optimal WSOS lower bound of $\vt$. The analysis also motivates the stopping criterion for the algorithm.

In what follows, we let $\lambda_{\max}(\mathbf{M})$  denote the largest eigenvalue of the matrix $\mathbf{M}$ and $\lambda_{\min}(\mathbf{M})$ denote the smallest eigenvalue. We also remark that $\|\cdot\|_1$,  $\|\cdot\|$ and $\|\cdot\|_\infty$ refer to the standard 1-norm, 2-norm, and infinity norm of vectors, respectively (not to be confused with the local norms used above).
\begin{theorem}\label{thm:lincon} Suppose that $\vt - c^*\vone$ is on the boundary of $\Sigma$. Let $\vy$ denote the gradient certificate of some $\vt-c\vone$ with $c<c^*$. Then there exists a constant $C$ (depending only on the operator $\Lambda$) such that $c^*-c \leq (C\|\vone\|_\vy^*)^{-1}$.
\end{theorem} 
\begin{proof}
Recall that $-g(\vy) = \vt - c\vone$. Define the constant \[k_1 \defeq \min\{\vone^\T\vv \ | \ \vv \in \Sigma^*, \|\vv\|_\infty = 1\}.\] Observe that the minimum exists (as $\Sigma^*$ is a closed and non-trivial cone) and $k_1 > 0$, because $\vone \in \Sigma^\circ$. Using the shorthand $\alpha \defeq c^* - c > 0$, we now have
\begin{align*}
    \nu &\overset{\eqref{eq:gH-identities}}{=}\left  \langle -g\left(\frac{\vy}{\|\vy\|_\infty}\right), \frac{\vy}{\|\vy\|_\infty}\right \rangle \\
    &\overset{\eqref{eq:log-homogeneity}}{=} \|\vy\|_\infty \left\langle \vt - c\vone, \frac{\vy}{\|\vy\|_\infty} \right\rangle \\
    &= \|\vy\|_\infty\left(\left\langle \vt - c^*\vone, \frac{\vy}{\|\vy\|_\infty}, \right\rangle + (c^*-c) \left\langle \vone, \frac{\vy}{\|\vy\|_\infty}\right\rangle  \right) \\
    &\geq0 + \|\vy\|_\infty\alpha k_1 =  \|\vy\|_\infty\alpha k_1,
\end{align*}
from which we conclude that
\begin{equation}\label{eq:linconk1} 
\|\vy\|_\infty \leq \frac{\nu}{\alpha k_1}.
\end{equation}
Recall from Eq.~\eqref{eq:H} that $H(\vy)\vv = \Lambda^*(\Lambda(\vy)^{-1}\Lambda(\vv)\Lambda(\vy)^{-1})$. Therefore, $\vv^\T H(\vy)\vv = \langle \vv, \Lambda^*(\Lambda(\vy)^{-1}\Lambda(\vv)\Lambda(\vy)^{-1}) \rangle = \text{tr}(\Lambda(\vv)\Lambda(\vy)^{-1}\Lambda(\vv)\Lambda(\vy)^{-1})$. 
Moreover, observe that for every $\vA\succcurlyeq \vzero$ and real symmetric matrix $\vB$ of the same size, we have
\[
\text{tr}(\vA)\lambda_{\min}(\vB) \leq \text{tr}(\vA\vB) \leq \text{tr}(\vA)\lambda_{\max}(\vB).
\]
Using this fact, we have that for every $\vv\in\R^U$,
\begin{align*}
    \vv^TH(\vy)\vv &= \text{tr}\left(\Lambda(\vv)\Lambda(\vy)^{-1}\Lambda(\vv)\Lambda(\vy)^{-1}\right) \\
    &\geq\lambda_{\min}(\Lambda(\vy)^{-1})~\text{tr}\left(\Lambda(\vv)\Lambda(\vy)^{-1}\Lambda(\vv)\right)\\
    &=\lambda_{\min}(\Lambda(\vy)^{-1})~\text{tr}\left(\Lambda(\vv)^2\Lambda(\vy)^{-1}\right) \\
    &\geq\lambda_{\min}(\Lambda(\vy)^{-1})^2~\text{tr}(\Lambda(\vv)^2)\\
    &=\lambda_{\max}(\Lambda(\vy))^{-2}~\text{tr}(\Lambda(\vv)^2). 
\end{align*}
We conclude that

\begin{equation}\label{eq:c64.2}
\lambda_{\min}(H(\vy)^{1/2}) \geq \frac{k_2}{\lambda_{\max}(\Lambda(\vy))},
\end{equation}
wherein we define
\[k_2 \defeq \min\{\sqrt{\text{tr}(\Lambda(\vv)^2)} \ | \ \|\vv\| = 1\}.\]
We remark that $k_2 = \sigma_{\min}(\Lambda) > 0$ (since $\Lambda(\vv) \neq \vzero$ whenever $\vv \neq \vzero$).

Next, recall that $\|\vone\|_{\vy}^* = \|H(\vy)^{-1/2}\vone\|$ and note $\|H(\vy)^{-1/2} \| = \frac{1}{\lambda_{\min}\left(H(\vy)^{1/2}\right)}$. Define
\[
k_3 \defeq \max\left\{\lambda_{\max}(\Lambda(\vy))\ \middle|\ \vy\in\Sigma^*, \|\vy\|_\infty = 1\right\}.
\]
These identities and our previous inequalities give
\begin{equation*}
    \|\vone\|_{\vy}^* = 
    \|H(\vy)^{-1/2}\vone\| \leq 
    \frac{\|\vone\|}{\lambda_{\min}\left(H(\vy)^{1/2}\right)} \overset{\eqref{eq:c64.2}}{\leq}
    \frac{\lambda_{\max}(\Lambda(\vy))\|\vone\|}{k_2} \leq
    \frac{k_3\|\vy\|\|\vone\|}{k_2} \overset{\eqref{eq:linconk1}}{\leq} \frac{k_3\nu \|\vone\|}{k_1k_2\alpha}.
\end{equation*}
Defining $C \defeq \frac{k_1k_2}{k_3\nu \|\vone\|}$, we conclude that 
\[
\alpha = c^*-c \leq (C\|\vone\|_\vy^*)^{-1}. \qedhere
\]
\end{proof} 

We remark that the parameter $\nu=\sum_{i=1}^m L_i$ is a parameter of the WSOS cone $\Sigma$ entirely independent of the representation of the polynomials. The parameter $k_1$ depends on the basis in which the WSOS polynomials are represented (but otherwise does not depend on $\Lambda$), while $k_2$ and $k_3$ are properties of the $\Lambda$ operator representing $\Sigma$.

Coupling \mbox{Lemma \ref{thm:m-m*}} with \mbox{Theorem \ref{thm:lincon}}, we have also proven our main result about the convergence of our algorithm: 

\begin{theorem}\label{thm:linear-convergence} Algorithm \ref{alg:Newton} is globally linearly convergent to
$c^* = \max\{c \ | \ \vt - c\vone \in \Sigma\}$, the optimal WSOS lower bound for the polynomial $\vt$. More precisely, in each iteration of Algorithm \ref{alg:Newton}, the improvement of the lower bound $\Delta c = c_+ - c$ satisfies
\begin{equation}\label{eq:distance-bound}
\frac{\Delta c}{c^* - c} \geq \rho_r C,
\end{equation}
with the absolute constant $\rho_r>0$ defined in Lemma \ref{thm:m-m*} and the $\Lambda$-dependent constant $C>0$ defined in Theorem \ref{thm:lincon}.
\end{theorem} 

Theorem \ref{thm:linear-convergence} motivates the stopping criterion (\mbox{Line \ref{line:stopping}}) of \mbox{Algorithm \ref{alg:Newton}}: the current bound $c$ is guaranteed to satisfy $c\leq c^* \leq c+\varepsilon$ as soon as $\Delta c \leq \rho_r C\varepsilon$.

Alternatively, we can rearrange the same inequality to provide an explicit upper bound on the number of iterations of the algorithm. After $k$ iterations of \mbox{Algorithm \ref{alg:Newton}} we have
\[
c^* - c_k \leq (1 - \rho_rC)^k(c^* - c_0),
\]
therefore, for a fixed cone (and parameter $C$), the algorithm terminates after
$ \Oh\left(\log\frac{c^*-c_0}{\varepsilon}\right) $
iterations. Additionally, it is typically easy to bound from above the global minimum of the input polynomial $\vt$ (e.g., by evaluating it at any point in its domain), and thus bound $c^*$ from above, and when an explicit bound on the magnitude of the elements in $\{\vx\in\R^n\,|\,w_i(\vx)\geq0,\,i=1,\dots,m\}$
is known, it is also straightforward to upper bound $c^*$ by $\kappa_\vw\|\vt\|$ with some constant $\kappa_\vw$ dependent only the weight functions $\vw$. Similarly, from the first step of Algorithm~\ref{alg:Newton} \revision{ (with $\vx \in \cC(\vone)$), 
\begin{align*}
c_0 &= -\left(\frac{r}{r+1} - \|-g(\vx) - \vone\|_{\vx}^*\right)^{-1}\|\vt\|^*_{\vx}\\ 
&\geq -\left(\frac{r}{r+1} - \|-g(\vx) - \vone\|_{\vx}^*\right)^{-1} \lambda_{\max}(H(\vx)^{-1})\|\vt\|,
\end{align*}
bounding the initial bound $c_0$ from below by a $\Lambda$-dependent constant multiple of $\|\vt\|$.}

In conclusion, for a fixed cone (and representation $\Lambda$), the algorithm terminates after $\Oh(\log\frac{\|\vt\|}{\varepsilon})$
iterations.

We also remark that although our primary goal is to obtain certified rational \emph{lower} bounds on the polynomial, dual certificates also provide \emph{upper bounds} on the optimal WSOS bound via Theorem~\ref{thm:linear-convergence}, whenever the $\Lambda$-dependent constant $C$ defined in the proof of Theorem \ref{thm:lincon} of is known (or can be bounded from below) for a particular cone $\Sigma$. In particular, although the analysis heavily relies on the quantity $\|\vone\|_\vy^*$, which is not efficiently computable (we do not have access to the gradient certificate $\vy$), the inequality \eqref{eq:distance-bound} provides a computable upper bound on $c^*$.

\revision{
\paragraph{The effect of finite precision} As all iterative numerical algorithms, \mbox{Algorithm \ref{alg:Newton}} will eventually fail to make progress before satisfying its stopping criterion if implemented in finite precision and with a  tolerance $\varepsilon$ that is too small. As $c$ approaches the optimal WSOS lower bound, $\vt-c\vone$ and $\vx$ approach the boundary of $\Sigma$ and $\Sigma^*$, respectively. This is marked by the increasing ill-conditioning of the Hessian $H(\vx)$ used in the Newton step in Line \ref{line:x-update} and of the quadratic equation solved in Line \ref{line:c-update}. Thus, the algorithm will stall if the Hessian is numerically singular or if $c_+$ does not improve on $c$ due to rounding errors in the quadratic formula.

We emphasize that this only affects the quality of the bound (how close we can get to the optimal WSOS lower bound), not the correctness of the bounds that are claimed to be certified. The validity of the certificate $\vx$ certifying the current lower bound $c$ can be verified in rational arithmetic at the end of any iteration; there is no danger of the numerical method yielding an incorrect bound or an invalid certificate undetected. Comparing Lemma \ref{thm:constant-update} with Corollary \ref{thm:x-ycert}, we see that as long as the iterates of Algorithm \ref{alg:Newton} satisfy $\|\vx-\vy\|_\vx < 1/2$, the certificate and the current bound will remain valid, and that with infinite-precision computation, these iterates would even satisfy $\|\vx-\vy\|_\vx \leq 1/4$. As long as the numerical errors are small enough that the iterates remain in the $1/2$-radius local norm ball (instead of the expected $1/4$-radius ball), the algorithm computes certifiable bounds and rigorous rational certificates in spite of every step of the computation being imprecise.
}

\subsection{Bounding constants in Theorem \ref{thm:lincon}} \label{sec:C-bounds}
In general, we cannot hope to find a sharp closed-form bound for the constant $C$ in Theorems \ref{thm:lincon} and \ref{thm:linear-convergence}, but we can compute cone-specific bounds on each of the constants $k_1, k_2,$ and $k_3$ in the formula for $C$ by convex optimization.

Recall that $k_1 = \text{min}\{\vone^\T\vv \ | \ \vv \in \Sigma^*, \|\vv\|_\infty = 1\}$. Although the norm constraint is not convex, we have
\begin{equation*}
k_1 = \min_{1 \leq i \leq U} \{\min\{k_{1,i}^-, k_{1,i}^+\}\},
\end{equation*}
with
\begin{equation}\label{eq:k1+}
k_{1,i}^+ = \min\{\vone^\T\vv \ | \ \vv \in \Sigma^*, \|\vv\|_\infty \leq 1\text{ and } v_i = 1\} \quad (i=1,\dots,m)
\end{equation}
and
\begin{equation}\label{eq:k1-}
k_{1,i}^- = \min\{\vone^\T\vv \ | \ \vv \in \Sigma^*, \|\vv\|_\infty \leq 1\text{ and } v_i = -1 \}. \quad (i=1,\dots,m)
\end{equation}
Therefore, $k_1$ can be computed (numerically) by solving $2U$ convex optimization problems. (For a rigorous lower bound, we can use dual methods that determine approximately optimal but feasible solutions of the dual optimization problems of \eqref{eq:k1+} and \eqref{eq:k1-}.)

Recall that $k_2 = \min\{\text{tr}(\Lambda(\vv)^2) \ | \ \|\vv\| = 1\}$. Hence, the constant $k_2$ is the smallest singular value of the linear operator $\Lambda$ and can be computed to high accuracy using singular value decomposition. Alternatively, we have
\begin{equation*}
    \text{tr}(\Lambda(\vv)^2) = \sum_{i=1}^L \Lambda_i(\vv)^2 = \vv^\T \vM \vv,
\end{equation*} for a positive semidefinite rational matrix $\vM$ that is easily computable from $\Lambda$; lower bounding $k_2$ amounts to lower bounding the smallest eigenvalue of the matrix $\vM$. 

Recall that the constant $k_3  = \max\left\{\lambda_{\max}(\Lambda(\vy))\ \middle|\ \vy\in\Sigma^*, \|\vy\|_\infty = 1\right\}$. 
Using the Gershgorin circle theorem, we know that 
\begin{equation}\label{eq:k3.1}
\lambda_{\max}(\Lambda(\vy))
= \max_{1 \leq k \leq m} \lambda_{\max}(\Lambda_k(\vy))
\leq \max_{1 \leq k\leq m} \|\Lambda_k(\vy)\|_\infty.
\end{equation}
So $k_3$ can be bounded from above by the largest absolute row sum of all of the $\Lambda_k$ operators. 

Since the values of $\|\vone\|$ and $\nu$ are known, having bounded $k_1$ and $k_2$ from below by positive quantities and $k_3$ from above, $C$ can be bounded from below by a positive, efficiently computable constant. In Section \ref{sec:univariate} we revisit this question and find closed-form bounds for the case of univariate nonnegative polynomials over an interval.

\revision{
\subsection{Assumptions}\label{sec:assumptions}
Throughout, we have made two fundamental assumptions. The first is that the constant one polynomial is in the interior of the WSOS cone $\Sigma = \Sigma_{n,2\vd}^\vw$. (Naturally, in any remotely interesting situation, positive constant polynomials must belong to $\Sigma$, but not necessarily to the interior.) The second is that we have access to \emph{some} dual certificate of the constant one polynomial in $\Sigma$. These assumptions are mild and practically unrestrictive, as we shall discuss below.

\paragraph{Constant one polynomial in the interior of $\Sigma$}
This is a mild assumption both from a theoretical and practical perspective. In many cases (when the weights are sufficiently simple), it can be verified directly and ensured to hold a priori; it can also be verified via convex optimization. If $\vone$ is only on the boundary, there are various ways to expand $\Sigma$ to a larger WSOS cone that will contain $\vone$ in the interior: first, as we discussed in Section \ref{sec:rational-existence}, as long as the assumptions of Putinar's Positivstellensatz are satisfied, every positive polynomial on $S_\vw$ is in the interior of $\Sigma_{n,2\vd}^\vw$ for every sufficiently large degree vector $\vd$. Alternatively, $\Sigma=\Sigma_{n,2\vd}^\vw$ can be extended, with the inclusion of a single additional weight that is nonnegative on $S_\vw$, to satisfy this condition without changing $\operatorname{span}(\Sigma)$ (in particular, without increasing the degrees), as stated in the following theorem.

\begin{theorem}
Suppose $\Sigma_{n,2\vd}^\vw \subseteq \rr^U$ is full dimensional. Let $\vr$ be the coefficient vector of a polynomial which is bounded (positively) away from zero on $S_\vw$. Then $\vr$ is in the interior of  $\Sigma_{n,2(d_1,\dots,d_m,0)}^{(w_1,\dots,w_m,w_{m+1})}$, for some weight polynomial $w_{m+1}$ which is nonnegative on $S_\vw$. 
\end{theorem}

\begin{proof} Since $\Sigma_{n,2\vd}^\vw \subseteq \rr^U$ is full dimensional, we can select linearly independent coefficient vectors $\vs_1,\vs_2,\dots,\vs_{U} \in \Sigma_{n,2\vd}^\vw$ such that $\vr \not\in \text{aff}(\{\vs_1,\vs_2,\dots,\vs_{U}\})$. (Here, $\text{aff}$ denotes the affine hull.) Define ${\vs_{U+1} := M\vr - \sum_{i=1}^{U}\vs_i}$, with $M \in \rr$ large enough to guarantee that $\vs_{U+1} \geq 0$ on $S_\vw$ and that $\vs_{U+1}\not\in \text{aff}(\{\vs_1,\vs_2,\dots,\vs_{U}\})$. By construction, $\vr =  \frac{1}{M}\sum_{i=1}^{U} \vs_i$, that is, $\vr$ is in the interior of the simplex whose vertices are the (affinely independent) vectors $\frac{M}{U+1}\vs_i$, $i=1,\dots,U+1$. Each of these vertices belong to the WSOS cone $\Sigma_{n,2\hat{\vd}}^{\hat{\vw}}$, where $\hat{\vw}=(w_1,\dots,w_m, w_{m+1})$ is the set of initial weights augmented with the new weight polynomial $w_{m+1}$ whose coefficient vector is $\vs_{U+1}$ and $\hat{\vd}=(d_1,\dots,d_m, 0)$. (The new weight $w_{m+1}$ is only multiplied by nonnegative constants in this new WSOS cone.) Therefore, $\vr$ is in the interior of $\Sigma_{n,2\hat{\vd}}^{\hat{\vw}}$.
\end{proof} 


\paragraph{Certificate for the constant one polynomial in Algorithm \ref{alg:Newton}} We have also assumed that we have access to a certificate $\vx \in \cC(\vone)$. In the examples in this paper, we could determine the gradient certificate for $\vone$, $\vxo$, in closed form. If we do not know $\vxo$ explicitly, then a crude approximation of $\vx_1$ is already sufficient to initialize \mbox{Algorithm \ref{alg:Newton}}. Precisely, we need to compute a certificate $\vx \in \cC(\vone)$ satisfying \mbox{$\|-g(x) - \vone\|_{\vx} \leq \frac{r}{r+1}$} (wherein the parameter $r$ here is the same as that in \mbox{Algorithm \ref{alg:Newton}}). There are several approaches to compute such a vector. For instance, we can use any convex optimization algorithm to minimize the convex, self-concordant function $\overline{f}(\vx) \defeq \vone^\T\vx + f(\vx)$ (with $f(\vx)$ defined in \eqref{eq:def:f}); the minimizer of $\overline{f}(\vx)$ satisfies $-g(\vx) = \vone$. As we only need a certificate $\vx$ with $-g(\vx)$ in the neighborhood of $\vone$, an approximate minimizer of $\overline{f}(\vx)$ returned by a numerical optimization method is sufficient.

Alternatively, motivated by two-phase interior-point methods that start from an approximate analytic center, we can run Algorithm \ref{alg:Newton} ``in reverse'' to find a certificate of $\vone$ as long as we have access to \emph{any} vector in $\Ssc$. The intuition is that if some vector $\vx$ certifies $\vt+c\vone$, then it also certifies $c^{-1}\vt+\vone$, which is approximately the same polynomial as $\vone$ if $c$ is sufficiently large.

Thus, starting with any vector $\vx \in \Ssc$ and its corresponding polynomial $\vt = -g(\vx) \in \Sc$, we iterate a modified certificate update step and a modified constant update step which are identical to those in Lines \ref{line:x-update} and \ref{line:c-update} of the algorithm, except that we replace the polynomial $\vt-c\vone$ with $\vt+c\vone$, in order to find a \emph{large} constant $c$ and a certificate $\vx$ for which $\vx$ certifies $\vt+c\vone$. We terminate this iterative process when $\|\vt + c\vone - c\vone\|_{\vx} = \|\vt\|_{\vx} < r'$ is sufficiently small to guarantee that $\|(\vt+c\vone)/c - \vone\|_{c\vx}$ is small enough to ensure that $c\vx$ certifies $\vone$. We omit the details of this analysis that are analogous to the analysis of Algorithm \ref{alg:Newton}; in particular, the proof of its rate of convergence can be adapted to show that in this ``reversed'' algorithm, $c$ increases exponentially, and $\|(\vt+c\vone)/c - \vone\|_{c\vx}$ converges to zero at a linear rate.
}

\section{Univariate polynomials}\label{sec:univariate}
In the univariate case, we can bound the number of iterations of Algorithm \ref{alg:Newton} by providing explicit bounds on the constant $C$, adapting the arguments from those in Section \ref{sec:C-bounds}. For brevity, we only treat the even-degree case in detail.
\begin{theorem}\label{thm:univariateChebyshev}
Suppose that $n=1$ and $\deg t = 2d$. Using the Chebyshev basis to to represent all polynomials and weights $\vw(z) \defeq (1, 1 - z^2)$ (as in Example \ref{ex:Chebyshev}), \mbox{Algorithm \ref{alg:Newton}} terminates after at most $\mathcal{O}(d^2\log\frac{\|\vt\|d}{\varepsilon})$ iterations and requires $\mathcal{O}(d^5\log\frac{\|\vt\|d}{\varepsilon})$ floating point operations overall.
\end{theorem} 
\begin{proof} 
We start by bounding the constant $C$ from Theorem \ref{thm:lincon} as a function of all relevant parameters by bounding each of $k_1,k_2$ and $k_3$ in the formula for $C$.
\begin{enumerate}
    \item $k_1 \geq 1$. Recall that $k_1 = \min \{  \vone^\T \vv \ | \ \vv \in \Sigma^*, \|\vv\|_\infty = 1\}$. 
Since nonnegative polynomials and weighted sum-of-squares polynomials coincide in the univariate case \cite{BrickmanSteinberg1962}, every vector $\vv \in \left(\Sigma_{1,2d}^\vw \right)^*$ can be written as a conic combination of moment vectors; precisely, we write $\vv = \sum_{i=1}^n \alpha_i \vq(z_i)$, wherein $z_i \in [-1,1]$ and $\alpha_i \geq 0$ for each $i$ \cite[Sec.~II.2]{KarlinStudden1966}. Then, we have \[
\vone^\T\vv = \vone^\T\left(\sum_{i=1}^n \alpha_i \vq(z_i) \right) = \sum_{i=1}^n \alpha_i \left(\vone^\T \vq(z_i)\right) = \sum_{i=1}^n\alpha_i.
\]
If $\|\vv\|_\infty = 1$, there exists some $j$ such that $|\sum_{i=1}^n \alpha_i \vq_j(z_i)| = 1$. Since for the Chebyshev basis each $\vq(z_i) \in [-1,1]^{2d+1}$, it follows that 
\[
\vv_j = 1= \left|\sum_{i=1}^n \alpha_i \vq_j(z_i)\right| \leq \sum_{i=1}^n \left| \alpha_i \vq_j(z_i) \right| \leq \sum_{i=1}^n |\alpha_i| = \sum_{i=1}^n \alpha_i, 
\]
since $\alpha_i \geq 0$. Thus, $\sum_{i=1}^n \alpha_i \geq 1$. It follows that 
\[
\vone^\T\vv = \sum_{i=1}^n \alpha_i \geq 1, 
\]
therefore $k_1 \geq 1$.

\item $k_2 \geq \frac{1}{2}\sqrt{3 - \sqrt{5}} \approx 0.437$. Recall that $k_2 = \min \{ \text{tr}(\Lambda(\vv)^2) \ | \ \|\vv\| = 1\}$. 
We have
\[
 \text{tr}(\Lambda(\vv)^2) = \text{tr}(\Lambda_1(\vv)^2) + \text{tr}(\Lambda_2(\vv)^2) \geq \text{tr}(\Lambda_1(\vv)^2).
\]
Note that
\begin{equation}\label{eq:TiTj}
2p_i(x)p_j(x) =  p_{i+j}(x) + p_{|i - j|}(x) \quad \text{ for every }i,j=0,1,\dots
\end{equation}
Coupling this identity with the fact that $\Lambda_1(\vq) = w_1\vp\vp^\T= \vp\vp^\T$ (recall that the first weight is $w_1=1$), we deduce that
\begin{equation}\label{eq:TiTjLambda}
\Lambda_1(\vv)_{i,j} = \frac{1}{2}v_{i + j} + \frac{1}{2}v_{|i - j|}.
\end{equation}
Therefore, the zeroth row (and the zeroth column) of $\Lambda_1(\vv)$ is $(v_0,v_1,\dots,v_d)$ and the last row (and the last column) is $(\frac{1}{2}(v_d + v_d), \frac{1}{2}(v_{d-1} + v_{d+1}), \dots, \frac{1}{2}(v_0 + v_{2d}))$, and so we have 
\[
\text{tr}(\Lambda_1(\vv))^2 \geq  \sum_{i=0}^d v_i^2 + \sum_{j=0}^d \frac{1}{4}(v_{j} + v_{2d - j})^2 = \vv^\T\vM\vv,
\]
where $\vM$ is the $2d + 1 \times 2d + 1$ matrix (indexed from zero) given by
\[ M_{i,j} = \begin{cases}
                    \frac{5}{4} & \text{ if } i=j < d  \\
                    2 & \text{ if } i=j=d \\
                    \frac{1}{4} & \text{ if } i + j = 2d, i\neq j,\text{ or if }i=j > d\\
                    0 &\text{ otherwise }
                 \end{cases} \]
Therefore
\[
\vv^\T \vM \vv  \geq \|\vv\|^2 \lambda_{\min}(\vM) = \|\vv\|^2 \left(\frac{1}{4}(3 - \sqrt{5})\right). 
\]

We conclude that 
\[
k_2  \geq \frac{1}{2}\sqrt{3 - \sqrt{5}} \approx 0.437. 
\]

\item $k_3 \leq d+1$.  Recall  that
\[k_3  = \max\{\lambda_{\max}(\Lambda(\vy)) \ |\  \vy \in \Sigma^*, \|\vy\|_\infty = 1\},\] and that based on the inequality \eqref{eq:k3.1}, we need only bound the largest absolute row sum of $\Lambda_1(\vy)$ and $\Lambda_2(\vy)$, \revision{ for all $\vy \in \Sigma$ with $\|\vy\|_{\infty} = 1$}. 

For $\Lambda_1(\vy)$, the identity \eqref{eq:TiTjLambda} and $\|\vy\|_\infty = 1$ yield the bound
\[
\sum_{j=0}^d |\Lambda_1(\vy)_{i,j}| = \sum_{j=0}^d\left|\frac{1}{2}\vy_{i+j} + \frac{1}{2}\vy_{|i - j|}\right| \leq d +1.
\]

For $\Lambda_2(\vy)$, observe that $1 - t^2 = \frac{1}{2}(p_0(t) - p_2(t))$.  Coupling this with the identity \eqref{eq:TiTj}, we deduce that
\begin{align*}
\frac{1}{2}(p_0(t) - p_2(t))p_i(t)p_j(t) 
=& \frac{1}{8} \big((2 p_{i + j}(t) + 2 p_{|i - j|} (t) - p_{i + j + 2}(t) - \\
   & p_{|i + j - 2|}(t) - p_{|i - j| + 2}(t)-p_{||i - j| - 2|}(t)\big),
\end{align*}
so
\[
\Lambda_2(\vy)_{i,j} = \frac{1}{8} \left(2 \vy_{i + j} + 2\vy_{|i - j|} - \vy _{i + j + 2} -
   \vy_{|i + j- 2|} - \vy_{|i - j| + 2} - \vy_{||i -j| - 2|}\right).
\]
Then, assuming $\|\vy\|_\infty = 1$, we obtain the bound
\[
    \sum_{j=0}^d |\Lambda_2(\vy)_{i,j}|  \leq \sum_{j=0}^d \left|\frac{1}{8} \left(2+ 2 +1 + 
  1 + 1+ 1\right)\right| \leq d+1.
\]

Thus, $k_3 \leq \max\{\max\{\|\Lambda_1(\vy)\|_\infty, \|\Lambda_2(\vy)\|_\infty\} \ |\  \|\vy\|_\infty = 1\} \leq  d+1$.
\end{enumerate}

Lastly, since $\nu = 2d + 1$ and $\|\vone\| = 1$, combining the above bounds on $k_1$, $k_2$ and $k_3$ we get
\[
C = \frac{k_1k_2}{k_3\nu\|\vone\|} \geq \frac{\sqrt{3 - \sqrt{5}}}{2(d+1)(2d+1)}.
\]

From \eqref{eq:distance-bound} and the discussion directly following Theorem \ref{thm:linear-convergence}, the number of iterations is proportional to
\begin{equation}\label{eq:number-of-iterations}
\mathcal{O}\left(\frac{1}{\log\frac{1}{1-\rho_r C}}\log\frac{c^*-c_0}{\varepsilon}\right)
\end{equation}
\revision{(recalling that $\varepsilon$ is the inputted tolerance from Algorithm \ref{alg:Newton}). }
From the series expansion $\frac{1}{\log\frac{1}{1-z}} = z^{-1}-\frac{1}{2}-\dots$ we see that the first term in \eqref{eq:number-of-iterations}, which only depends on the input through the degree $d$, is $\Oh(C^{-1}) = \Oh(d^2)$. To bound the numerator of second term, recall that for a coefficient vector $\vt$ in the Chebyshev basis, $|c^*|\leq \|\vt\|_1 \leq (2d+1)^{1/2}\|\vt\|_2$, and from the intialization of \mbox{Algorithm \ref{alg:Newton}} we also have
\[|c_0| \leq \frac{1+r}{r}\|H(\vx_1)^{-1/2}\|_2\|\vt\|_2 \overset{\eqref{eq:Hx1}}{\leq} \frac{1+r}{r}\|\vt\|_2\frac{2d+1}{\sqrt{2}}. \]
Thus, $|c^*-c_0|$ is of order $\Oh(\|\vt\|_2 d)$, and the claim about the number of iterations follows.

The bottleneck of each iteration is the computation and factorization of the Hessian, which require $\mathcal{O}(d^3)$ floating point operations. Therefore, the total number of floating point operations is $\mathcal{O}(d^5\log\frac{\|\vt\|d}{\varepsilon})$. The $\Oh(\cdot)$ notation hides only absolute constants and the user-defined constant parameter $\rho_r$ from Lemma \ref{thm:m-m*}.
\end{proof}
\alittlelessspace
\revision{
\subsection{Optimization versus certification}
Our approach and analysis have been motivated from the perspective of \emph{optimization}, where the goal is to compute a certified lower bound as close to the global minimum as possible. From this perspective, the dependence of the complexity of \mbox{Algorithm \ref{alg:Newton}} on the parameter $\varepsilon$, or rather on $\|\vt\|/\varepsilon$, which measures the relative error of the lower bound is arguably one of the most important questions. In this section, we interpret the above results from another perspective, that of computational algebraic geometry and symbolic computing, where the related fundamental question is often posed as follows: given a polynomial of integer coefficients (in the monomial basis) that is known to be positive on a given domain, what is the complexity of certifying its positivity? The answer, in principle, ought be a function of the number of variables, the degree, and the bit size of the coefficients. Note that $\varepsilon$ is not a relevant parameter in this question, as the polynomial is assumed to be positive.

Polynomial-time certification of nonnegativity is challenging even in the univariate case. 
The methods of Guo et~al.~\cite{GuoSafeyElDinZhi2013} and Schweighofer \cite{Schweighofer1999} are exponential in the degree.
(The complexity of the latter was only recently established in \cite{MagronSafeyElDinSchweighofer2019}.) On the other hand, \cite{BoudaoudCarusoRoy2007} presents an algorithm for computing positivity certificates of a polynomial over an interval in time polynomial in both the degree and the bit size of the coefficients. Another algorithm, published by Chevillard et~al.~\cite{ChevillardHarrisonJoldesLauter2011} and analyzed in  \cite{MagronSafeyElDinSchweighofer2019} is also polynomial in the degree, and so is the recent algorithm of Magron and Safey El Din \cite{MagronSafeyElDin2018}. Of the methods mentioned in this paragraph, all those that have polynomial complexity in the degree in the univariate case rely on techniques that are exclusive to univariate polynomials and cannot be generalized to the multivariate case.

Neither of these complexity results are directly comparable to our results above, as we are concerned with nonnegativity over compact sets (as opposed to the real line, as most of the papers cited above), we are working in the real number model for floating-point computation, and to avoid the inherent numerical instability of high-degree polynomials represented in the monomial basis, we consider the input to be a coefficient vector in the Chebyshev basis. But \mbox{Algorithm \ref{alg:Newton}} can be adapted to the above decision problem as follows: given a polynomial by its (integer) coefficient vector $\vt$ in the Chebyshev basis, we can compute a positive lower bound $\mu$ on its minimum over the interval $[-1,1]$ assuming that this minimum is positive, using a technique of Basu, Leroy, and Roy  \cite{BasuLeroyRoy2009}. Finally, we can invoke \mbox{Algorithm \ref{alg:Newton}} with tolerance $\epsilon < \mu$. As soon as the certified lower bound turns positive, the algorithm can be terminated.

Theorem 1.2 of \cite{BasuLeroyRoy2009} gives a positive lower bound on the minimum of a polynomial with integer coefficients in the monomial basis over the interval $[0,1]$. This bound is a function of only the degree $d$ of the polynomial and the maximum bitsize $\tau$ of its coefficients. The proof can be adapted to bounding the minimum over $[-1,1]$ without any substantial changes. The change of basis (from the Chebyshev basis to monomial) can be incorporated using the observation that a degree-$d$ polynomial with integer coefficients of bit size at most $\tau$ in the Chebyshev basis also has integer coefficients in the monomial basis, and the bit size of the largest magnitude coefficient is no more than $2d+\tau$. Thus we have the following bound.
\begin{lemma}[\cite{BasuLeroyRoy2009}; Thm.~1.2, adapted]
Let $t$ be a univariate polynomial of degree $d$ taking only positive values on the interval $[-1,1]$, and suppose that the coefficients of $t$ in the Chebyshev basis are integers of bit size no more than $\tau$. Then we have
\[ \min_{z\in[-1,1]} t(z) >
\frac{ 3^{d/2} }{ 2^{(2d-1)(2d+\tau)} (d+1)^{2d-1/2} } =: \mu(d,\tau). \]
\end{lemma}
Hence, the number of iterations of \mbox{Algorithm \ref{alg:Newton}} with tolerance $\epsilon=\mu(d,\tau)$ is
\[ \Oh\left(d^2\log\frac{\|\vt\|d}{\mu(d,\tau)} \right) = 
\Oh\left(d^4+d^3\tau\right),
\]
polynomial in the degree and linear in the bit size of the coefficients.

We underline that, unlike the algorithms mentioned above, \mbox{Algorithm \ref{alg:Newton}} does not rely on any techniques that are specific to univariate polynomials. Its complexity in the multivariate setting is a subject of future research.
}

\revision{
\section{Numerical examples}\label{sec:examples}
In this section, we report the results of numerical experiments with a simple Mathematica implementation of \mbox{Algorithm \ref{alg:Newton}} and investigate the quality of the best certifiable rational lower bound using the purely numerical version of the algorithm implemented using double-precision floating point arithmetic. These results are summarized in \mbox{Table \ref{tab:experiment}}.

\paragraph{Problem instances} Problems 1--7 are standard benchmark problems from the polynomial optimization literature going back to at least \cite{RayNataraj2009} (but see also, e.g., \cite{MunozNarkawicz2012,MurrayChandrasekharanShah2020,PappYildiz2019}). Problem 8 is from \cite{MunozNarkawicz2012}, Problem 9 is Example 4 from \cite{PeyrlParrilo2008}. The problem is originally unconstrained; we added the constraint $\|\vz\|\leq 10$ to the problem, which can be shown to be redundant. All polynomials were represented in the monomial basis throughout the computation. The minimum values of these polynomials are known.

\paragraph{Implementation details} The analysis of the method provides theoretically safe choices for the algorithmic parameters; in the experiments we used $r=1/4$. Initial points were determined either from simple closed-form solutions of the nonlinear system $-g(\vx)=\vone$ or approximate numerical solutions, as discussed in Section \ref{sec:assumptions}. Instead of the stopping criterion used in the theoretical analysis, we ran the method with $\varepsilon=0$, until numerical issues prevented progress, to compute the best possible bound and a corresponding rational dual certificate.

The timing results were obtained on a standard Macbook laptop computer equipped with 16GB RAM and a 2.8 GHz Intel Core i7 processor with 4 cores running macOS 10.15.7. We used Mathematica version 12.0.0.0.

The $c^*-c$ column of \mbox{Table \ref{tab:experiment}} shows the difference between the minimum and the certified lower bound returned by the algorithm.

Following the numerical computation, we also used the approach outlined in \mbox{Section \ref{sec:c-update}} to determine an approximately optimal lower bound certified by the dual certificate returned by \mbox{Algorithm \ref{alg:Newton}}. Denoting the known optimal value of the polynomial by $c^*$, a simple linear search can identify the smallest (negative) integer exponent $k$ for which the lower bound $c^*-10^{k}$ is certified by the same dual certificate. The value of this $k$ is reported in the last column of \mbox{Table \ref{tab:experiment}}.

\paragraph{Summary} As expected, the numerical method yields bounds with roughly 5--8 correct decimal digits, and as predicted by the theory, the numerical dual certificate can certify stronger bounds than those returned by the Algorithm \ref{alg:Newton}. It is remarkable, however, that the best bounds are often at least 5-6 digits more accurate than the ones from the algorithm (this is owing to the fact that the algorithm uses the simple sufficient condition in the $c$-update steps, rather than looking for a near-optimal bound), and that in one example, the numerical dual certificate can even certify a bound that is \emph{indistinguishable from the minimum} with the precision used throughout the computation. (Both the relative and absolute errors of the bound are far smaller than the unit round-off in double precision.)
  }

\setlength{\textfloatsep}{0pt}
\begin{table}\label{tab:experiment}
\centering
\revision{
\begin{tabular}{llllrrlr}
\toprule
\multicolumn{1}{c}{\#} & \multicolumn{1}{c}{Problem} & \multicolumn{1}{c}{$n$}  & \multicolumn{1}{c}{$d$} & \# iters & time {[}s{]} & \multicolumn{1}{c}{$c^*-c$}            & $k$\\
\midrule
1 & R. D. 3   & $3$ & $2$ & 121        & 0.11     & $2.690981304 \times 10^{-6}$  & $-22$ \\
2 & Schwefel  & $3$ & $4$ & 384        & 1.08     & $5.764365051\times10^{-7}$ & $-13$ \\
3 & L. V. 4   & $4$ & $3$ & 467        & 2.24     & $2.602585946\times10^{-5}$ & $-11$ \\
4 & Caprasse  & $4$ & $4$ & 283        & 2.45     & $2.260781469\times10^{-6}$ & $-10$ \\
5 & Butcher   & $6$ & $3$ & 268        & 8.66     & $1.180076686\times10^{-6}$ & $-13$  \\
6 & Magnetism & 7   & 2   & 318        & 1.16     & $9.031997478\times10^{-8}$ & $-15$ \\
7 & Heart Dipole & 8 & 4   & 374      & 127.9     & $8.688025884\times10^{-6}$ & $-7$\\
8 & Wrig 5    & $5$ & $2$ & 172        & 0.34     & $1.339771877\times10^{-6}$ & $-12$ \\
9 & PP Ex 4   & $3$ & $4$ & 332        & 0.58     & $1.397125980\times10^{-6}$ & $-12$ \\
\bottomrule\\
\end{tabular}
\caption{Summary of results of running Algorithm \ref{alg:Newton} on standard benchmark problems from \cite{RayNataraj2009,MunozNarkawicz2012,PeyrlParrilo2008}. $n$ and $d$ represent the number of variables and the degree of the polynomial to be minimized. The $c^*-c$ column lists the difference between the certified lower bound $c$ returned by \mbox{Algorithm \ref{alg:Newton}} and the true minimum $c^*$ of the polynomial. The $k$ column lists the smallest integer exponent $k$ for which $c^* - 10^{k}$ can be certified by the outputted certificate vector from Algorithm \ref{alg:Newton}.}
}
\end{table}

\section{Discussion}\label{sec:Discussion}
\paragraph{Primal versus dual certificates} 
Conventional nonnegativity certificates are representations of the certified polynomials that make their nonnegativity apparent. This is a fundamental issue for numerical methods for computing nonnegativity certificates, as the certificate they compute is typically a rigorous WSOS certificate for a slightly different polynomial from the one we seek to certify.

Dual certificates address this issue: through the formula \eqref{eq:Sdef}, not only can we interpret any rational dual vector from $\cC(\vs)$ as a certificate, but we can also compute, via a closed-form formula, a rational certificate for the polynomial $\vs$ with rational coefficients. Since every polynomial (in the interior of the SOS cone) has a full-dimensional cone of dual certificates, even an inexact numerical method computing low-accuracy solutions to an SOS optimization problem can return dual certificates that can be turned into a rational certificate this way. For example, Algorithm \ref{alg:Newton} can be implemented as a purely numerical method, followed by an application of the formula \eqref{eq:Sdef} to compute a rational certificate for the computed bound. Although the certificate $\vx$ only loosely tracks the gradient certificate of $\vt-c\vone$, we can guarantee that $\vx$ certifies the current bound. This also means that, unlike most numerical or hybrid methods that require high-accuracy solutions from the numerical component of the algorithm, Algorithm \ref{alg:Newton} provides a certified bound even if terminated early; only the quality of the bound suffers.

Recent work in numerical methods for non-symmetric cones has resulted in a few additional algorithms that can directly optimize over the cone of WSOS certificates circumventing semidefinite programming, including \cite{KarimiTuncel2019} and \cite{PappYildiz2020}; in principle, these can also be coupled with the methods presented in Section \ref{sec:dualcertificates}.

\paragraph{Efficiency} In general,
it is difficult to make general statements about the asymptotic running time of \mbox{Algorithm \ref{alg:Newton}} as a function of every parameter (the degree and the number of unknowns of the input polynomial, etc.) as these also depend on the specific weight polynomials and the chosen representation ($\Lambda$ operator). As noted, the computational cost per iteration is a low-degree polynomial for $\Lambda$ operators corresponding to popular bases in numerical methods (e.g., Chebyshev and interpolant bases), and the method is linearly convergent, that is, for a given polynomial it requires a number of iterations proportional to $\log(1/\varepsilon)$ to compute a certified rational bound within $\varepsilon$ of the optimal bound $c^*$. 
We derived an explicit bound on the linear rate and the initial gap $c^* - c_0$ in the univariate case in Section \ref{sec:univariate}; it may also be possible to derive such bounds in other important special cases, such as the cases of multivariate polynomials over simple semialgebraic sets such as the unit sphere or the unit cube. 

\revision{\paragraph{Application to polynomials with particular structures}
Dual certificates can be used in combination with other recent approaches aimed at increasing the practical efficiency of sum-of-squares optimization, such as exploiting sparsity or symmetry. For example, the (term- or correlative-) sparsity of the polynomial in \cite{KojimaKimWaki2005, Lasserre2006, WangMagronLasserre2021-2} has the effect of making the $\Lambda$ operator simpler, either by reducing the dimension of $\Lambda(\vx)$ or by imposing a block structure on it, reducing the overall size of the problem at hand. This in turn simplifies the computation of $\Lambda$ and $H^{-1}$ and the verification of the semidefiniteness of $\Lambda(\vx)$, for $\vx \in \Sigma^*$. Future research could also examine how exploiting symmetries, e.g.,~following \cite{RienerTheobaldAndrenLasserre2012}, may be used in tandem with dual certificates. Likewise, future work could extend dual certificates to the noncommutative setting.}
\alittlelessspace

\bibliographystyle{siamplain}     
\bibliography{mega}

\end{document}